\documentclass[12pt]{amsart}
\usepackage{epsfig,color}

\usepackage{hyperref}
\hypersetup{
    colorlinks=true, 
    linktoc=all,     
    linkcolor=blue,  
}

\usepackage{mathrsfs}
\usepackage{amssymb}

\theoremstyle{definition}

\newtheorem{theorem}{Theorem}[section]

\newtheorem{proposition}[theorem]{Proposition}
\newtheorem{lemma}[theorem]{Lemma}
\newtheorem{remark}[theorem]{Remark}
\newtheorem{corollary}[theorem]{Corollary}

\numberwithin{equation}{section}

\usepackage{enumerate}

\DeclareMathOperator*{\spa}{span}

\renewcommand\div{\operatorname{div}}

\DeclareMathOperator*{\id}{id}

\newcommand{\e}{\operatorname{e}}
\newcommand{\wt}{\widetilde}
\newcommand{\pr}{\partial}

\newcommand{\Lap}{\Delta}

\newcommand{\Ric}{\operatorname{Ric}}

\renewcommand*\d{\mathop{}\!\mathrm{d}}

\DeclareMathOperator*{\diam}{diam}

\DeclareMathOperator*{\spec}{spec}

\usepackage{comment}

\title[{\L}ojasiewicz inequalities for cylindrical self-shrinkers]{{\L}ojasiewicz inequalities, uniqueness and rigidity for cylindrical self-shrinkers}

\author{Jonathan J. Zhu}
\address{Mathematical Sciences Institute, Australian National University, Hanna Neumann Building, Science Road, Canberra, ACT 2601, Australia and Department of Mathematics, Princeton University, Fine Hall, Washington Road, Princeton, NJ 08544, USA}
\email{jjzhu@math.princeton.edu}

\begin{document}
\begin{abstract}
We establish {\L}ojasiewicz inequalities for a class of cylindrical self-shrinkers for the mean curvature flow, which includes round cylinders and cylinders over Abresch-Langer curves, in any codimension. We deduce the uniqueness of blowups at singularities modelled on this class of cylinders, and that any such cylinder is isolated in the space of self-shrinkers. The Abresch-Langer case answers a conjecture of Colding-Minicozzi. Our proof uses direct perturbative analysis of the shrinker mean curvature, so it is new even for round cylinders. 
\end{abstract}
\date{\today}
\maketitle

\section{Introduction}
Self-shrinkers (or just shrinkers) are submanifolds $\Sigma^n \subset \mathbb{R}^N$ that satisfy the elliptic PDE
\begin{equation}
\label{eq:shrinker}
\phi := \frac{x^\perp}{2} - \mathbf{H} =0,
\end{equation}
where $x^\perp$ is the normal projection of the position vector, and $\mathbf{H}$ is the mean curvature vector. Shrinkers correspond to homothetically shrinking solutions of the mean curvature flow (MCF), which describes the motion of submanifolds $\Sigma_t^n \subset \mathbb{R}^N$ by the parabolic PDE
\begin{equation}
(\pr_t x)^\perp = -\mathbf{H}. %
\end{equation}

Shrinkers also model singularities of MCF, arising as (subsequential) rescaling limits, or blowups. Uniqueness of a blowup (that is, independence from the sequence) is a fundamental question in singularity analysis, which for MCF is intertwined with the question of whether a shrinker is rigid, or isolated in the space of shrinkers. Uniqueness and rigidity for round cylinders were major open problems, recently proven by Colding and Minicozzi \cite{CM15, CM19, CMcomp} together with Ilmanen \cite{CIM}, and have many powerful consequences for generic MCF \cite{CMsurvey}. 

In this article, we define a class of `simply non-integrable' generalised shrinking cylinders $\mathring{\Gamma}^k \times \mathbb{R}^{n-k}$, and show that they are rigid, and give unique tangent flows if $\mathring{\Gamma}$ is embedded. Simply non-integrable cylinders have tame Jacobi fields, and this class includes round cylinders $\mathbb{S}^k_{\sqrt{2k}}\times \mathbb{R}^{n-k}$, and cylinders $\mathring{\Gamma}_{a,b}^1 \times \mathbb{R}^{n-1}$ over Abresch-Langer curves $\mathring{\Gamma}_{a,b}$ \cite{AL}. 

\begin{theorem}[Uniqueness of tangent flows]
\label{thm:unique}
Let $\mathring{\Gamma}^k \subset \mathbb{R}^{k+m}$ be a closed, simply non-integrable shrinker and let $M_t^n$ be a MCF in $\mathbb{R}^N$. Assume either that $\mathring{\Gamma}$ is embedded, or that $M_t$ has a type I singularity at $(x_0,t_0)$. If one tangent flow at $(x_0,t_0)$ is induced by $\Gamma =\mathring{\Gamma}\times \mathbb{R}^{n-k}$ with multiplicity one, then all tangent flows $(x_0,t_0)$ are given by $\Gamma$ (with multiplicity one, and no rotation). 
\end{theorem}

\begin{theorem}[Rigidity]
\label{thm:rigidity}
Let $\mathring{\Gamma}^k \subset \mathbb{R}^{k+m}$ be a closed, simply non-integrable shrinker and let $\Sigma^n\subset \mathbb{R}^N$ be another shrinker. Assume either that $\mathring{\Gamma}$ is embedded, or $\sup_\Sigma |A_\Sigma|^2 \leq C_1$. There exists $R_1=R_1(\lambda_0,C_1)>0$ such that if $\Sigma$ has entropy $\lambda(\Sigma)\leq \lambda_0$ and $\Sigma\cap B_R$ is a graph $U$ over $\Gamma =\mathring{\Gamma}\times \mathbb{R}^{n-k}$ for some $R\geq R_1$, and $\|U\|_{C^{2,\alpha}} \leq R^{-1}$, then $\Sigma$ is a rotation of $\Gamma$. 
\end{theorem}

Simple non-integrability is defined in Section \ref{sec:assumptions} by certain conditions on the cross-section $\mathring{\Gamma}$, which again are satisfied by $\mathbb{S}^k_{\sqrt{2k}}$ and $\mathring{\Gamma}_{a,b}$. The Abresch-Langer case of Theorem \ref{thm:rigidity} answers a conjecture of Colding-Minicozzi \cite[Section 6.4]{CMcomp}, and both theorems provide new, conceptually simpler proofs for the case of round cylinders. Recall that a MCF has a type I singularity at time $t_0$ if $\sup_{t<t_0} \sup_{M_t} \sqrt{t_0-t} |A_{M_t}| <\infty$. 


The entropy of a shrinker coincides with its Gaussian area \[\textstyle F(\Sigma^n) =\int_\Sigma \rho \d \mu_\Sigma,\qquad \text{ where }\rho=  (4\pi)^{-n/2}  \e^{-|x|^2/4}.\] The $L^2$-gradient of $F$ is precisely the shrinker quantity $\phi$, which characterises shrinkers as the critical points of $F$. The key as in \cite{CM15, CM19} to proving Theorems \ref{thm:unique} and \ref{thm:rigidity} is to establish {\L}ojasiewicz inequalities for generalised shrinking cylinders $\mathring{\Gamma} \times \mathbb{R}^{n-k}$. Let $\mathcal{C}_n(\mathring{\Gamma})$ denote the set of all rotations of $\mathring{\Gamma} \times \mathbb{R}^{n-k}$ about the origin in $\mathbb{R}^N$. Our main theorems are the following {\L}ojasiewicz inequalities (see also Theorem \ref{thm:lojasiewicz}; cf. \cite[Theorems 0.24, 0.26]{CM15}):

\begin{theorem}[{\L}ojasiewicz inequality of the first kind]
\label{thm:lojasiewicz-intro} 
 Let $\mathring{\Gamma}$ be a closed, simply non-integrable shrinker.
Given $q\in(1,\frac{4}{3}), \beta\in(\frac{1}{q},1)$, there exist $\epsilon_2>0, l$ so that the following holds: For any $\epsilon_1$, $\lambda_0$, $C_j$ there is an $R_0$ such that if $\Sigma^n\subset \mathbb{R}^N$ has $\lambda(\Sigma)\leq \lambda_0$ and
\begin{enumerate}
\item for some $R>R_0$, we have that $B_R\cap \Sigma$ is the graph of a normal field $U$ over some cylinder in $\mathcal{C}_n(\mathring{\Gamma})$ with $\|U\|_{C^{2}(B_R)} \leq \epsilon_2$ and $\|U\|_{L^2(B_R)}^2 \leq \e^{-R^2/8}$, 
\item $|\nabla^j A| \leq C_j$ on $B_R\cap \Sigma$ for all $j \leq l$;
\end{enumerate}
then there is a cylinder $\Gamma \in\mathcal{C}_n(\mathring{\Gamma})$ and a compactly supported normal vector field $V$ over $\Gamma$ with $\|V\|_{C^{2,\alpha}} \leq \epsilon_1$, such that $\Sigma \cap B_{R-6}$ is contained in the graph of $V$, and 

\[ \|V\|_{L^2}^2 \leq C(\|\phi\|_{L^q}+\|\phi\|_{L^2}^{2\beta} + (R-5)^{\beta n} \e^{-\beta \frac{(R-5)^2}{4}}) ,\]
where $C=C(n,\beta,q,C_j,\lambda_0, \epsilon_1)$. 
\end{theorem}

Theorem \ref{thm:lojasiewicz-intro} bounds the $L^2$ distance to $\mathcal{C}_n(\mathring{\Gamma})$ by $\|\phi\|_{L^q}$, where $q$ may be taken arbitrarily close to $1$, plus a small exponential error term and a higher order $\|\phi\|_{L^2}$ term. 

\begin{theorem}[{\L}ojasiewicz inequality of the second kind]
\label{thm:lojasiewicz-grad} 
 Let $\mathring{\Gamma}$ be a closed, simply non-integrable shrinker. 
There exist $\epsilon_2>0, l$ so that for any $\lambda_0$, $C_j$ there is an $R_0$ such that if $l\geq l_0$, $\Sigma^n\subset \mathbb{R}^N$ has $\lambda(\Sigma)\leq \lambda_0$ and
\begin{enumerate}
\item for some $R>R_0$, we have that $B_R\cap \Sigma$ is the graph of a normal field $U$ over some cylinder in $\mathcal{C}_n(\mathring{\Gamma})$ with $\|U\|_{C^{2}(B_R)} \leq \epsilon_2$ and $\|U\|_{L^2(B_R)}^2 \leq \e^{-R^2/8}$, 
\item $|\nabla^j A| \leq C_j$ on $B_R\cap \Sigma$ for all $j \leq l$;
\end{enumerate}
then for some $C=C(n,C_j,\lambda_0)$ we have 

\[ |F(\Sigma) - F(\mathring{\Gamma})| \leq C(\|\phi\|_{L^2}^\frac{3}{2} + (R-5)^{\frac{3n}{4}} \e^{-\frac{3(R-5)^2}{16}}) .\]
\end{theorem}
Note that $F(\Gamma) = F(\mathring{\Gamma})$ for any cylinder $\Gamma \in \mathcal{C}_n(\mathring{\Gamma})$. As typically expected, the {\L}ojasiewicz inequality of the first kind (`distance {\L}ojasiewicz') implies the second (`gradient {\L}ojasiewicz'). The celebrated {\L}ojasiewicz-Simon \cite{simon83} inequality is of the second kind. The uniqueness and rigidity Theorems \ref{thm:unique} and \ref{thm:rigidity} are deduced using the iterated extension/improvement schemes developed by Colding-Minicozzi \cite{CM15, CM19} and Colding-Ilmanen-Minicozzi \cite{CIM}. In particular, Theorem \ref{thm:lojasiewicz-intro} supplies the `improvement step', while we use the `extension step' essentially as proven in \cite{CIM, CM15, CM19}. 

A key feature of our proof of Theorem \ref{thm:lojasiewicz-intro} is that we use only perturbative analysis and Taylor expansion of $\phi$ (to order 2, corresponding to the exponent of $\|V\|_{L^2}$), inspired in part by \cite{ELS18}. By contrast, for round cylinders Colding-Minicozzi \cite{CM15, CM19} use a geometric argument relying on the auxiliary quantity $\tau = \frac{A}{|\mathbf{H}|}$, which is only well-defined if $\mathbf{H}$ is nonvanishing. Our analysis is thus more general, and it also gives the optimal exponents on $\|\phi\|$. 

\subsection{Simply non-integrable shrinkers}
\label{sec:assumptions}

Let $A$ be the second fundamental form, which takes values in the normal bundle, so that $\mathbf{H}$ is the negative trace of $A$. 

We say that a shrinker $\mathring{\Gamma}^k \subset \mathbb{R}^{k+m}$ is \textit{simply non-integrable} if it satisfies the following conditions (A1-A2), as well as $B_1(\mathring{\Gamma})\neq 0$. (Note that our convention is that $\lambda$ is an eigenvalue of $L$ with eigenvector $V$ if $LV= -\lambda V$, and similarly for $\mathcal{L}$.) 

\begin{itemize}
\item[(A1)] The drift Laplace operator $\mathcal{L} = \Lap - \frac{1}{2} \nabla_{x^\top}$ on $\mathring{\Gamma}$, acting on functions has $\frac{1}{2}$-eigenspace spanned by coordinate functions, \[\textstyle\ker (\mathcal{L}+\frac{1}{2}) = \spa \{x_i\}_{1\leq i\leq m} .\]

\item[(A2)] The Jacobi operator $L = \mathcal{L} + \frac{1}{2} + \sum_{i,j} \langle \cdot, A_{ij}\rangle A_{ij}$ on $\mathring{\Gamma}$, acting on normal vector fields, satisfies $\spec L \cap -\frac{1}{2}\mathbb{N} \subset \{0,-\frac{1}{2}, -1\}$; moreover the corresponding eigenspaces are spanned by infinitesimal rotations, translations and dilations respectively;
\begin{alignat*}{3}
 &\ker L &&= \spa\{ x_i\pr_{x_j}^\perp - x_j \pr_{x_i}^\perp\}_{1\leq i,j\leq k+m},\\
 & \textstyle\ker (L - \frac{1}{2}) &&= \spa\{\pr_{x_i}^\perp \}_{1\leq i\leq k+m},\\
 &\ker (L-1) &&= \spa\{ x^\perp \}= \spa\{ \mathbf{H}\}. 
\end{alignat*}
\end{itemize}

The quantity $B_1$ is defined by:

\begin{equation} B_1(\mathring{\Gamma}) = \int_{\mathring{\Gamma}} \left( A^\mathbf{H}_{ij} A^\mathbf{H}_{il} A^\mathbf{H}_{lj} - 3 A^\mathbf{H}_{ij} \langle \nabla^\perp_i \mathbf{H}, \nabla_j^\perp \mathbf{H}\rangle  \right)\rho.  \end{equation}

It is straightforward to check that (A1-A2) hold and $B_1(\mathring{\Gamma})<0$ on both round spheres and Abresch-Langer curves (see Section \ref{sec:prelim}). Note that these are conditions on the cross-section but non-integrability refers to the cylinder $\mathring{\Gamma}\times \mathbb{R}^{n-k}$; (A1-A2) imply that space of Jacobi fields on $\mathring{\Gamma}\times \mathbb{R}^{n-k}$ is minimal (Section \ref{sec:analysis1}) and $B_1$ arises in the variation of $\phi$ (Section \ref{sec:analysis2}). 

\subsection{Background}

Uniqueness of blowups is one of the most important questions for the regularity theory of geometric PDEs. It has been studied in a wide range of contexts, to which we defer to \cite{CM15} and references therein. For MCF, a widely used notion of blowup is the tangent flow: One considers a sequence of parabolic dilations about a fixed spacetime point $(x_0,t_0)$ - for instance, about $(x_0,t_0)=(0,0)$ we take $\lim_{\lambda_i\to\infty} \lambda_i \cdot M_{t/\lambda_i^2}$. Every tangent flow is induced by a self-shrinker $\Gamma$ and we consider two important cases: when $\Gamma$ is smoothly embedded with multiplicity one, and when the singularity is type I. In these cases, convergence to $\Gamma$ is locally smooth (by Brakke's regularity theorem, or the work of Huisken \cite{Hui90}, respectively). 
Schulze \cite{S14} established uniqueness of tangent flows at all closed shrinkers. For general noncompact tangent flows, the first results were the uniqueness at round cylinders due to Colding-Minicozzi \cite{CM15, CM19}, which yielded a series of powerful consequences for MCF with cylindrical singularities \cite{CMsing, CMdiff, CMreg, CMarnold}. Recently, Chodosh-Schulze \cite{CS19} proved uniqueness at asymptotically conical shrinkers. 

By rigidity of a solution to a geometric PDE, we mean that it is isolated in the space of solutions, modulo geometric symmetries of the equation. Colding-Ilmanen-Minicozzi \cite{CIM} proved a `strong rigidity' for the round cylinder (see also \cite{CMcomp, GZ18}), corresponding to the topology of locally smooth convergence. If the shrinker is only immersed, we also assume a global curvature bound, corresponding to a type I blowup. Rigidity of a shrinker corresponds to uniqueness of the shrinker \textit{type} of a tangent flow (leaving open the possibility of rotations). Rigidity of the Clifford torus $\mathbb{S}^1_{\sqrt{2}}\times\mathbb{S}^1_{\sqrt{2}} \subset \mathbb{R}^4$ was proven by Evans-Lotay-Schulze \cite{ELS18}, which was one of the inspirations for this paper; the author and A. Sun \cite{SZ20} have also extended this to all products $\mathbb{S}^{k}_{\sqrt{2k}} \times \mathbb{S}^l_{\sqrt{2l}}$. In his thesis, Adams \cite{adams} also discussed similar rigidities in particular directions. We also mention some global rigidity theorems \cite{S05, CL13, DX14} for shrinkers.

{\L}ojasiewicz inequalities are an increasingly ubiquitous tool to study the (dynamical) stability of geometric structures. L. Simon \cite{simon83} developed a general method based on Lyapunov-Schmidt reduction to classical inequalities of {\L}ojasiewicz \cite{Loj}; inequalities proven via this method are often called {\L}ojasiewicz-Simon inequalities; see for instance \cite{CMein,H12, HM14}, the survey \cite{CMsurvey} and the recent innovations of Feehan \cite{F19,F20}. For MCF, Schulze \cite{S14} and Chodosh-Schulze \cite{CS19} showed that {\L}ojasiewicz-Simon inequalities apply for the cases of closed and asymptotically conical shrinkers respectively. For round cylinders, however, Colding-Minicozzi \cite{CM15, CM19} proved {\L}ojasiewicz inequalities by a direct geometric method. In this paper, we also prove {\L}ojasiewicz inequalities directly, but using a novel perturbation method, which precisely identifies the relevant geometry of the model shrinker. During the completion of this work, we learnt that Colding-Minicozzi \cite{CMrf} are developing a similar theory for Ricci soliton cylinders. Kr\"{o}ncke \cite{Kr15} also initiated an analogous study for Ricci solitons. 


\subsection{Explicit {\L}ojasiewicz inequalities and deformation theory}
\label{sec:deformation}

Here we illustrate our expansion-deformation method for the zeroes of an analytic functional, inspired by \cite{ELS18} and also \cite{CM19}. Consider the model case of an analytic map $G:\mathbb{R}^n\to\mathbb{R}^n$, with $G(0)=0$. For instance, one may take $G= Df$ where $f:\mathbb{R}^n \to \mathbb{R}$ is an analytic function. We are interested in the zero set of $G$ near the origin. By Taylor expansion, for small enough $u$ we have
\begin{equation}
G(u) = \sum_{k=1}^\infty \frac{1}{k!} (D^k G)_0 (u,\cdots ,u). 
\end{equation}

If the linearisation $L=(DG)_0$ is invertible, then 0 is an isolated root and $|u| \sim |G(u)|$. Otherwise, let $\mathcal{K}=\ker L$ and write $u=x+y$ where $x\in \mathcal{K}$ and $y\in \mathcal{K}^\perp$. Consider $g\in \mathbb{R}^n$; then the map $(g,x,y)\mapsto \pi_{\mathcal{K}^\perp}( G(x+y)-g)$ is linearly invertible in $y$, so the implicit function theorem gives a function $y(g,x)$ so that $\pi_{\mathcal{K}^\perp}G(x + y(g,x))=\pi_{\mathcal{K}^\perp} (g)$. To understand the zero set of $G$, we can analyse the solution $y$ at each degree in $x$.

Indeed, as in the proof of the (analytic) implicit function theorem, we can expand $y(0,x)$ as a series $y(0,x) = \sum y_j(x)$, where each $y_j\in\mathcal{K}^\perp$ is a homogenous polynomial of degree $j$ in $x$. Expanding the equation $\pi_{\mathcal{K}^\perp}G(x+y(0,x))=0$ and comparing degrees, we find that \[Ly_1=0, \qquad \text{(hence $y_1=0$)},\] \[Ly_2 + \pi_{\mathcal{K}^\perp} (z_2) =0, \qquad z_2:= \frac{1}{2}(D^2G)_0(x,x), \] \[Ly_3 +\pi_{\mathcal{K}^\perp}(z_3)=0, \qquad z_3 :=\frac{1}{6}(D^3G)_0(x,x,x) + (D^2G)_0(x,y_2(x)),\] and so forth, inductively defining $y_j,z_j$. Then $G(x+y(0,x)) = \pi_\mathcal{K} G(x+y(0,x)) =\sum \pi_{\mathcal{K}} z_j(x)$, which gives a simple criterion for the zero set to be (uniformly) obstructed: 
\begin{itemize}
\item[($\dagger$)] There exists $m$ such that $\pi_{\mathcal{K}}(z_j)=0$ for all $j<m$, but $\pi_{\mathcal{K}}(z_m(x))\neq 0$ for all $x$. 
\end{itemize}

Let $w_k = y-\sum_{j\leq k} y_k$. Using the expansion of $G(u) = G(x+ y_1 +\cdots +y_k + w_k)$ we have \begin{equation}\label{eq:model-exp}
G(u) = Lw_k+\sum_{j\leq k} (Ly_j+z_j)  +O(|u|^{k+1}+|x|^{k+1}+|x||w_{k}|).\end{equation}
If ($\dagger$) holds then for $k<m$ we have $Ly_k+z_k=0$, hence 
\begin{equation}
\label{eq:model-est}
|w_k| \lesssim |Lw_k|\lesssim |G(u)| + |u|^{k+1} ,\qquad \pi_{\mathcal{K}} z_m = \pi_{\mathcal{K}} G(u) + O(|u|^{m+1}). 
\end{equation} 
But ($\dagger$) also implies $|\pi_{\mathcal{K}}(z_m(x))| \geq \delta |x|^m$ for some $\delta>0$. It follows that $ |x|^m \lesssim |G(u)| + |u|^{m+1}$ and therefore \begin{equation}|u|^m \lesssim |G(u)|.\end{equation} 
This inequality shows that the origin is an isolated zero, and gives a {\L}ojasiewicz inequality for $f$ if $G=Df$, which is explicit in that it identifies the zero set as well as the exponent. Note full analyticity is in fact not needed; we only used Taylor expansion of order $m$ with remainder, so it is enough for $G$ to be $C^{m+1}$. Also note that for $m=2$, ($\dagger$) is the single condition that $\pi_{\mathcal{K}} D^2G$ desfines a nondegenerate quadratic form on $\mathcal{K}$. 

More generally, suppose that the putative zero set is a submanifold, which may be assumed to be a subspace $\mathcal{K}_0 \subset \mathcal{K}$ without loss of generality. In this case, the final estimate should control $|u-\pi_{\mathcal{K}_0}(u)|$. We mention this because in the infinite-dimensional, noncompact setting in this paper, $\mathcal{K}_0$ corresponds to the space of rotations, and does create technical issues.

In this paper, we find a uniform obstruction at order $m=2$, assuming (A1-A2) and $B_1(\mathring{\Gamma})\neq 0$. See Sections \ref{sec:analysis} and \ref{sec:entire}, wherein the roles of $G, u, x,y=w_1$ will be played by $\phi, U, J, h$ respectively. We expect that the same strategy may work on more general cylinders, although the analysis may be complicated. Note that we indeed use the same strategy in \cite{SZ20} to prove {\L}ojasiewicz inequalities of obstruction order $m=3$ for Clifford shrinkers. 

\subsection{Technical overview}

There are a number of technical issues in implementing the above scheme, arising predominantly due to the noncompactness of the model shrinker. We work with Sobolev spaces weighted by the Gaussian $\rho$ as in \cite[Section 6]{CM19}; in \cite{ELS18, SZ20} H\"{o}lder norms interact conveniently with Taylor expansion, but in the noncompact setting may require carefully defined weighted spaces (together with Sobolev spaces) as in \cite{CS19}. 

The main issue with Sobolev norms is that Taylor expansion increases the degree $m$ to which each quantity appears, with additional complications due to noncompactness. We are able to work in the natural $L^2$ space until the final, obstructed order, by applying a refinement of the techniques in \cite[Section 6]{CM19} to obtain $L^{2m}$ bounds. At this order, the degree becomes too high to control in $L^2$ and we instead establish the obstruction in $L^q$, $q\in(1,2)$. (See Section \ref{sec:entire}.) Actually, due to cutoff errors from noncompactness, it is advantageous to choose smaller $q$, although the noncompactness itself seems to prevent us from taking $q=1$. 

The formal expansion scheme is also complicated by the symmetries of the system (\ref{eq:shrinker}), that is, ambient rotations of the model shrinker $\Gamma$. If $\Gamma$ is compact, the implicit function theorem guarantees that any small perturbation may be written, after rotating the base, as a graph orthogonal to rotations. For noncompact cylinders, we instead rotate $\Gamma$ by hand; the rotated graph is then almost-orthogonal up to higher order terms. The $L^2$ assumption on $U$ in Theorem \ref{thm:lojasiewicz-intro} ensures that the applied rotation is small. (See Section \ref{sec:rotation}.) By contrast, in \cite{CM15, CM19} a good rotation is detected geometrically, so they do not need such an assumption. However, they only control the rotated graph on a domain dependent on its $L^2$ norm.

\subsection{Outline of the paper} 

In Section \ref{sec:prelim} we state our notation, conventions and some preliminary results. General variation formulae up to second order are then presented in Section \ref{sec:variation}. In Section \ref{sec:analysis}, we study variations of $\phi$ on cylinders, including the Jacobi fields and the formal second order obstruction. Following the expansion scheme, we deduce estimates for entire graphs in Section \ref{sec:entire}. In Section \ref{sec:rotation}, we describe the rotation of the base cylinder. In Section \ref{sec:lojasiewicz} we then combine these estimates to prove the {\L}ojasiewicz inequalities Theorem \ref{thm:lojasiewicz-intro} and Theorem \ref{thm:lojasiewicz-grad}. Finally, in Section \ref{sec:applications} we briefly describe how the rigidity and uniqueness of blowups follow from the {\L}ojasiewicz inequalities. We also include some interpolation inequalities in Appendix \ref{sec:interpolation}. 

\subsection*{Acknowledgements}

The author would like to thank Professors William Minicozzi, Tobias Colding and Felix Schulze for several enlightening discussions about their work, from which this project arose. The author also thanks Ao Sun, Sven Hirsch, Keaton Naff and Antoine Song for helpful conversations, and is grateful to Prof. Fernando Cod\'{a} Marques for his support. 
We would like to thank the anonymous referee for several comments that we feel have greatly improved the presentation of this article. 

This work was supported in part by the National Science Foundation under grant DMS-1802984 and the Australian Research Council under grant FL150100126. 


\section{Preliminaries}
\label{sec:prelim}

We consider smooth, properly immersed submanifolds $\Sigma^n\subset \mathbb{R}^N$. We use $x$ to denote the position vector on $\mathbb{R}^N$, although sometimes it will be fitting to denote a given immersion by $X: \Sigma\rightarrow \mathbb{R}^N$, particularly in Section \ref{sec:variation}. For instance, while $x(p)=X(p)$, we use $x_i$ to denote the coordinate functions whereas $X_i$ may refer to derivatives of the immersion $X$. 

For a vector $V$ we denote by $V^\top$ the projection to the tangent bundle, and $V^\perp = \Pi(V)$ the projection to the normal bundle $N\Sigma$. We similarly use the notation $\nabla^\perp_Y X = (\nabla_Y X)^\perp$. For instance, if $U$ is a normal vector field and $Y$ is a tangent vector field, the normal connection may indeed be calculated as $\nabla^\perp_Y U$. The Hessian on the normal bundle is given by $(\nabla^\perp\nabla^\perp V)(Y,Z) = \nabla^\perp_Z \nabla^\perp_Y V - \nabla^\perp_{\nabla^\top_Z Y} V$. 

On the other hand, if $Y,Z$ are tangent vector fields, then $A(Y,Z) = \nabla^\perp_Y Z$ defines the second fundamental form as a 2-tensor with values in the normal bundle. The mean curvature (vector) is $\mathbf{H} =  - A_{ii} $. 
Here, and henceforth, we take the convention that repeated lower indices are summed with the metric, for instance $A_{ii} = g^{ij} A_{ij}$. 
We introduce the shrinker mean curvature $\phi = \frac{1}{2}x^\perp-\mathbf{H}$ and the principal normal $\mathbf{N} = \frac{\mathbf{H}}{|\mathbf{H}|}$. A submanifold is a shrinker if $\phi\equiv0$ on $\Sigma$. Given a vector $V$ we denote $A^V = \langle A,V\rangle$. 

Given a vector field $U$ on $\Sigma$ with $\|U\|_{C^1}$ small enough, the graph $\Sigma_U$ is the submanifold given by the immersion $X_U(p)= X(p) + U(p)$. We say $\Sigma_U$ is a normal graph if $U^\top=0$. 

For graphs $\Gamma_U$ over a fixed submanifold $\Gamma$, we use subscripts to denote the values of geometric quantities on $\Gamma_U$. (Pulling back via the graph immersion, these may also be considered as quantities on $\Gamma$.) We also consider these quantities as second order functionals on normal vector fields $U$; namely, there is a smooth function $\mathcal{\varphi}$ such that $\phi_U =\mathcal{\varphi}(p,U,\nabla U,\nabla^2 U)$. For variations of this quantity, we use the shorthand notation $\mathcal{D}\mathcal{\varphi}(U)$ to mean the variation $\mathcal{D} \mathcal{\varphi}|_0 ([U,\nabla U,\nabla^2 U])$ evaluated at 0, and so forth. Here and throughout the paper, when $U$ is understood to be a normal vector field, we use the shorthand $\nabla=\nabla^\perp$ for the normal connection (note also that we will often work under uniform curvature estimates, in which case the ambient derivative differs only by lower order terms).

The Gaussian weight is $\rho=\rho_n=(4\pi)^{-n/2} \e^{-|x|^2/4}$. Here $n$ is understood to be the dimension of the submanifold and we suppress it when unambiguous. By $L^p$, $W^{k,p}$ we denote the weighted Sobolev spaces with respect to $\rho$; for instance, on $\Sigma$ the norms are
\begin{equation}
\label{eq:weighted-norm} \|\cdot\|_{L^p}^p = \int_\Sigma |\cdot|^p \rho , \qquad  \|\cdot\|_{W^{k,p}}^p = \int_\Sigma |\cdot|_k^p \rho,
\end{equation}
where $|\cdot|_{k} =\sum_{j\leq k} |\nabla^j(\cdot)|$. 
 The Gaussian area functional is $F(\Sigma) = \int_\Sigma \rho$. The entropy is $\lambda(\Sigma) = \sup_{y,s>0} F(s(\Sigma-y))$. For a shrinker, $\lambda(\Sigma)=F(\Sigma)$. Note that finite entropy $\lambda(\Sigma)\leq \lambda_0$ implies Euclidean volume growth $|\Sigma \cap B_R| \leq C(\lambda_0) R^n$. 

We will use the following elliptic operators: the drift Laplacian $\mathcal{L} = \Lap -\frac{1}{2}\nabla_{x^\top}$; and the Jacobi operator $L = \mathcal{L} +\frac{1}{2} + \sum_{k,l} \langle \cdot,A_{kl}\rangle A_{kl}$. The drift Laplacian is defined on functions and tensors, whilst $L$ is defined on sections of the normal bundle (via $\nabla^\perp$). For such operators, unless otherwise indicated, $\ker$ will refer to the $W^{2,2}$ kernel, for instance $\mathcal{K}= \ker L$. Again, our sign convention for eigenvalues $\lambda$ is $LV = -\lambda V$, and so forth. 

We set $\langle x\rangle = (1+|x|^2)^\frac{1}{2}$ and we refer to $2ab \leq \epsilon a^2 + \frac{1}{\epsilon} b^2$ as the absorbing inequality. 

In proofs, $C,C',\cdots$ will denote constants that may change from line to line but retain the stated dependencies. 

\subsection{Generalised cylinders}

Given a compact submanifold $\mathring{\Gamma}^k \subset \mathbb{R}^{k+m}$, the set of generalised cylinders $\mathcal{C}_n(\mathring{\Gamma})$ refers to all rotations of $\mathring{\Gamma} \times \mathbb{R}^{n-k} \subset \mathbb{R}^N$ about the origin. 

Decompose $\mathbb{R}^N = \mathbb{R}^{k+m} \times \mathbb{R}^{n-k} \times \mathbb{R}^{N-m-n}$ and let $\mathring{x}, y,z$ be the projection of $x$ to each respective factor. We typically use Latin indices $\mathring{x}_i$, $y_j$ for coordinates on $\mathbb{R}^{k+m}$ and $\mathbb{R}^{n-k}$, and Greek indices $z_\alpha$ for the remaining $N- m-n$ ambient directions. 

We use $\bar{\pi}$ to denote the projection to the linear directions $\mathbb{R}^{n-k}$, so that $y=\bar{\pi}(x)$. Similarly we use $\bar{\nabla}$ and $\bar{\mathcal{L}}$ to denote operators in the linear directions. 

\subsection{Shrinking spheres}

Round spheres $\mathbb{S}^k_{\sqrt{2k}} \subset\mathbb{R}^{k+1}$ are shrinkers when the radius is $\sqrt{2k}$. On the sphere we have $x^\top=0$, $\mathbf{H} = \frac{x}{2k}$ and $A_{ij} = - \frac{g_{ij}}{\sqrt{2k}}\mathbf{N}$. 

Thus the drift Laplace operator is simply the usual Laplacian, $\mathcal{L}=\Lap$. It is well known that $\spec \Lap = \{0, \frac{1}{2},\frac{k+1}{k}, \cdots\}$, with $\ker \Lap = \spa \{1\}$, $\ker (\Lap+\frac{1}{2}) = \spa\{ x_i \}_{1\leq i\leq k+1}$ and in particular $\ker (\Lap+1)=0$. 

Since the spheres have codimension 1, the normal bundle is trivial and the Jacobi operator acts on normal fields as $L(u\mathbf{N}) = (Lu)\mathbf{N}$. Here on the right hand side, $L$ is the Jacobi operator on functions given by $L = \Lap+1$. Since $\mathbf{H} = \frac{1}{\sqrt{2k}}\mathbf{N}$, $\pr_{x_i}^\perp = \frac{x_i}{\sqrt{2k}}\mathbf{N}$ and the rotation fields $x_i\pr_{x_j}^\perp - x_j \pr_{x_i}^\perp =0$ on $\mathbb{S}^k_{\sqrt{2k}}$, the above discussion implies that round shrinking spheres satisfy (A1-A2). Moreover $A^\mathbf{H}_{ij} = -\frac{g_{ij}}{2k}$, $\nabla^\perp \mathbf{H}=0$ and hence $B_1(\mathbb{S}^k_{\sqrt{2k}}) = -\frac{1}{8k^2} |\mathbb{S}^k_{\sqrt{2k}}|<0$. 

\subsection{Abresch-Langer curves}

The Abresch-Langer curves \cite{AL} are a family of smooth, closed, convex, immersed shrinkers $\Gamma_{a,b}^1\subset \mathbb{R}^2$. For such curves, we use $\sigma$ for the arclength parameter, with dots to denote differentiation in $\sigma$ so that $\mathbf{T} = \dot{x} = \pr_\sigma x$ is the unit tangent. The second fundamental form is $A_{ij} = -\kappa\mathbf{N} g_{ij}$, where the curvature $\kappa>0$ satisfies 

\begin{equation}\label{eq:AL}\ddot{\kappa} - \frac{\dot{\kappa}^2}{\kappa} + \kappa^3 = \frac{\kappa}{2}.\end{equation} 

The drift Laplace operator is given by $\mathcal{L} = \kappa\pr_\sigma(\kappa^{-1} \pr_\sigma)$. The operator $\mathcal{L}$ is a Sturm-Liouville operator, and by the standard theory of such operators the multiplicity of any eigenvalue is at most 2. The coordinate functions $x_1,x_2$ are independent eigenfunctions of $\mathcal{L}$ with eigenvalue $\frac{1}{2}$; therefore they span the corresponding eigenspace. In particular the Abresch-Langer curves satisfy (A1). 

Since the Abresch-Langer curves have codimension 1, again the normal bundle is trivial and the Jacobi operator acts on normal fields as $L(u\mathbf{N}) = (Lu)\mathbf{N}$. On the right hand side, $L$ is the Jacobi operator on functions, given by $L = \kappa\pr_\sigma(\kappa^{-1} \pr_\sigma) + \kappa^2+\frac{1}{2}$. In particular, the equation for the curvature is equivalent to $L\kappa=\kappa$, as expected. The operator $L$ is also a Sturm-Liouville operator, and since $\kappa>0$, we know that $\kappa$ is the lowest eigenfunction of $L$. Also the translations $\langle \pr_{x_1}, \mathbf{N}\rangle $, $\langle \pr_{x_2} , \mathbf{N}\rangle $ are independent eigenfunctions of $L$ with eigenvalue $-\frac{1}{2}$; again this implies that they span the corresponding eigenspace. Finally, it was checked by Baldauf-Sun \cite{BS18} that $0$ is a simple eigenvalue of $L$, with eigenfunction $\frac{\dot{\kappa}}{\kappa}$ corresponding to rotation. This discussion implies that the Abresch-Langer curves satisfy (A2). 

To check that $B_1(\Gamma_{a,b})<0$, we need the following lemma:

\begin{lemma}
For any $n\geq 1$ we have $\int_{\Gamma_{a,b}} \ddot{\kappa} \kappa^n = - n\int_{\Gamma_{a,b}} \dot{\kappa}^2 \kappa^{n-1}$. Consequently, for any $n\geq 2$ we have $-n \int_{\Gamma_{a,b}} \dot{\kappa}^2 \kappa^{n-2} = \int_{\Gamma_{a,b}}(\frac{1}{2}\kappa^n - \kappa^{n+2})$. 
\end{lemma}
\begin{proof}
The first identity is just integration by parts on the closed curve $\Gamma_{a,b}$. The second identity then follows by using equation (\ref{eq:AL}). 
\end{proof}

By \cite[Theorem A]{AL}, we have the relation $\rho = c_{a,b} \kappa^{-1}$ for some $c_{a,b}> 0$. Then by the lemma

\begin{equation}
B_1(\Gamma_{a,b}) = -\int_{\Gamma_{a,b}} (\kappa^6 - 3\kappa^2\dot{\kappa}^2)\rho = - c_{a,b} \int_{\Gamma_{a,b}}(\kappa^5-3\kappa\dot{\kappa}^2) = -\frac{c_{a,b}}{2} \int_{\Gamma_{a,b}} \kappa^3 <0.
\end{equation}

\subsection{Gaussian Poincar\'{e} inequality}

We will need the following lemma which exchanges growth for taking extra derivatives (cf. \cite[Lemma 3.4]{CM15}): 

\begin{lemma}
\label{lem:poincare}
Given a compact submanifold $\mathring{\Gamma}^k$ there is $C=C(n,\mathring{\Gamma})$ such that if $\Gamma =\mathring{\Gamma}\times \mathbb{R}^{n-k}$ and $u\in W^{1,2}(\Gamma)$ then 
\begin{equation}
\|\langle x\rangle u\|_{L^2}^2 \leq C(\|u\|_{L^2}^2 + \|\bar{\nabla} u\|_{L^2}^2) \leq C \|u\|_{W^{1,2}}^2.  
\end{equation}
\end{lemma}
\begin{proof}
We follow the proof of \cite[Lemma 3.4]{CM15}. Recall that $y=\bar{\pi}(x)$ is the component in the linear directions, so $y^\top=y$ and hence $\rho^{-1}\div_\Gamma(u^2 y \rho) = 2u\langle \nabla u,y\rangle +(n-k)u^2 - \frac{u^2}{2}\langle x^\top, y\rangle$. Since $\mathring{x}^\top$ is tangent to $\mathring{\Gamma}$ and hence orthogonal to $y$, we then have $\langle x^\top,y\rangle = |y|^2$ and after using the absorbing inequality, $\rho^{-1}\div_\Gamma(u^2 y \rho) \leq 4|\bar{\nabla} u|^2 + (n-k)u^2 - \frac{1}{4}u^2 |y|^2$. 

Approximating $u$ by smooth functions with compact support if necessary, by the divergence theorem we therefore have \[\frac{1}{4}\int_\Gamma u^2|y|^2 \rho \leq \int_\Gamma ((n-k)u^2 + 4|\bar{\nabla} u|^2)\rho.\] The lemma follows since $|x|^2 \leq |y|^2 + \diam(\mathring{\Gamma})^2$ on $\Gamma$. 
\end{proof}

Note the above applies equally well to vector fields. 

\subsection{Cutoff lemma}

We will need the following cutoff lemma, which is \cite[Lemma 7.16]{CM19}:

\begin{lemma}[\cite{CM19}]
\label{lem:cutoff}
Given $n,m\in\mathbb{N}$, there exists $c_{n,m}$ so that for any $R\geq 1$ we have 
\begin{equation}
\int_{\mathbb{R}^n \setminus B_R} |x|^m \e^{-|x|^2/4} \d x\leq c_{n,m} R^{n+m-2} \e^{-R^2/4}.
\end{equation}
\end{lemma}

The purpose of this lemma is to provide an estimate for geometric quantities outside a large compact set on which we have closeness to the model cylinder. 

\subsection{Rotation and Jacobi fields}
\label{sec:rot0}

We consider the action of the group $\mathrm{SO}(N)$ of rotations of $\mathbb{R}^n$ about the origin. In this subsection we use $\exp$ to denote the Lie group exponential $\mathfrak{so}(N) \rightarrow \mathrm{SO}(N)$. Let $\Gamma$ be a smooth shrinker in $\mathbb{R}^N$. Given $\theta \in \mathfrak{so}(N)$, we can consider the normal vector field on $\Gamma$ given by $J_\theta (p)  := \Pi( \pr_s |_{s=0} \exp(s\theta)(p))$. It is well known that $J_\theta$ satisfies the Jacobi equation $LJ_\theta=0$ since $\phi_{\exp(s\theta)\cdot\Gamma}\equiv 0$. The map $\theta\mapsto J_\theta$ is a linear map into the space of normal fields on $\Gamma$, and we denote its image by $\mathcal{K}_0$. 

We now identify the elements of $\mathcal{K}_0$ (see also \cite[Corollary 3.6 and Lemma 3.14]{CM19}). 

\begin{lemma}
\label{lem:rot-fields}
On a cylinder $\Gamma=\mathring{\Gamma} \times \mathbb{R}^{n-k}$, $\mathcal{K}_0$ is spanned by elements of the following forms:
\begin{enumerate}
\item $\mathring{x}_i\pr_{\mathring{x}_j}^\perp - \mathring{x}_j\pr_{\mathring{x}_i}^\perp$; 
\item $y_j \pr_{\mathring{x}_i}^\perp$; 
\item $\mathring{x}_i \pr_{z_\alpha}$; and $y_j \pr_{z_\alpha}$. 
\end{enumerate}
\end{lemma}
\begin{corollary}
\label{cor:rotation-est}
Let $\Gamma =\mathring{\Gamma} \times \mathbb{R}^{n-k}$. and set $r_0=\diam(\mathring{\Gamma})+1$. There exists $C$ depending on $N,\mathring{\Gamma}$ so that if $J \in \mathcal{K}_0$, then $|J| \leq C\langle x\rangle \|J\|_{L^2(B_{r_0})}$, $|\nabla J| + |\nabla^2 J| \leq C\langle x\rangle \|J\|_{L^2(B_{r_0})}$ and $|\nabla^2 J ( \cdot, \pr_{y_i})|  \leq C\|J\|_{L^2(B_{r_0})} $. 

In particular, $\mathcal{K}_0\subset W^{2,2}$. 
\end{corollary}
\begin{proof}
The group $\mathrm{SO}(N)$ is generated by rotations of 2 coordinates at a time. Rotation in the $x_ix_j$ plane is generated by the vector field $x_i \pr_{x_j} - x_j \pr_{x_i}$, whose normal projection is either zero or fits one of the above three types depending on whether $i,j$ correspond to coordinates on $\mathbb{R}^m\ni \mathring{x}$, the linear directions $\mathbb{R}^{n-k} \ni y$ or the extra dimensions $\mathbb{R}^{N-m-n} \ni z$. The estimates follow from the given forms, noting that $L^2(B_{r_0})$ gives a norm on $\mathcal{K}_0$. 
\end{proof}

Henceforth, we denote by $\pi_{\mathcal{K}_0}$ the $L^2$-projection to $\mathcal{K}_0$ and $\mathcal{K}_0^\perp$ the orthocomplement in $L^2$. We also fix a section $\iota$ identifying $\mathcal{K}_0$ with a subspace of $\mathfrak{so}(N)$, and fix metrics on $\mathfrak{so}(N)$ and $\mathrm{SO}(N)$ so that $\iota$ is an isometry and $\exp$ coincides with the Riemannian exponential, which is in particular a radial isometry.

\section{Variation of geometric quantities}
\label{sec:variation}

In this section we consider (normal) variations of a submanifold, and present the first and second variation of geometric quantities including in particular $A$ and $\phi = \frac{1}{2} X^\perp -\mathbf{H}$. 

Let $\Sigma$ be a submanifold with a fixed immersion $X_0 :\Sigma^n \rightarrow \mathbb{R}^N$, and a one-parameter family of immersions $X: I\times \Sigma^n\rightarrow \mathbb{R}^N$ with $X(0,p)=X_0(p)$. We use $s$ for the coordinate on $I=(-\epsilon,\epsilon)$, and subscripts to denote differentiation with respect to $s$. If $p_i$ are local coordinates on $\Sigma$, we get the tangent frame $X_i = X_*(\frac{\pr}{\pr p^i})$. 

All geometric quantities such as $\Pi, g,A$ should be considered as functions of $s,p$, given by the value of each quantity at $X(s,p)$ on the submanifold defined by $X(s,\cdot)$. For instance, the metric $g_{ij}(s,p)$ is given by $g_{ij}=\langle X_i, X_j\rangle$. Recall $\Pi$ is the projection to the normal bundle. Also recall our convention that repeated lower indices are contracted via the (inverse) metric $g^{ij}$, although we will raise indices when it suits the exposition. 

\subsection{First variation}

The first variation was calculated in \cite{CM19}:

\begin{proposition}[\cite{CM19}]
\label{prop:1stvar}
At $s=0$, suppose $X_s=V=V^\perp$; then we have

\begin{alignat}{3}
\label{eq:dPi}
&\Pi_s(W) &&=  -\Pi(\nabla_{W^\top} V)  - X_j g^{ij}\langle \Pi(\nabla_{X_i} V), W\rangle,\\
\label{eq:dg}
&(g_{ij})_s &&=  -2A^V_{ij}, (g^{ij})_s = 2g^{ik} A^V_{km} g^{mj}, \\
\label{eq:dA}
&(A_{ij})_s &&= -X_k \langle \nabla^\perp_{X_k} V, A_{ij}\rangle + (\nabla^\perp\nabla^\perp V)(X_i,X_j) - A^V_{ik} A_{jk},\\
 \label{eq:dphi}
 & \phi_s = \mathcal{D}\mathcal{\varphi} (V)&&= LV - X_j g^{ij} \langle \nabla^\perp_{X_i} V,\phi\rangle.
  \end{alignat}
\end{proposition}

\subsection{Second variation}

We now proceed to compute the second variation of $\phi$.

\begin{lemma}
\label{lem:d2Pi}
At $s=0$, assume $(X_s)^\top= (X_{ss})^\top=0$; then the second variation $\Pi_{ss}$ acts by:

\begin{alignat}{3}
\label{eq:d2PiTan}
&\Pi_{ss}(W^\top) &&= -\Pi(\nabla_{W^\top}X_{ss}) + 2\Pi(\nabla_{\nabla^\top_{W^\top} X_s} X_s) + 2 X_j g^{ij} \langle \Pi(\nabla_{X_i} X_s), \nabla_{W^\top} X_s\rangle,\\
\label{eq:d2PiNT}
&\langle X_i, \Pi_{ss}\Pi(W)\rangle &&= 2\langle W,\Pi(\nabla_{\nabla^\top_{X_i} X_s}X_s)\rangle - \langle W,\Pi(\nabla_{X_i} X_{ss})\rangle,\\
\label{eq:d2PiNN}
&\Pi \Pi_{ss}\Pi(W) &&= -2g^{ij}\langle W,\Pi(\nabla_{X_i} X_s)\rangle \Pi(\nabla_{X_j}X_s).
\end{alignat}

\end{lemma}
\begin{proof}
 Differentiating $\Pi^2=\Pi$ twice, we have $\Pi_{ss} \Pi + 2\Pi_s^2 + \Pi\Pi_{ss}= \Pi_{ss}$ and hence $\Pi \Pi_{ss} \Pi = -2\Pi_s^2 \Pi$. Then by (\ref{eq:dPi}), we have 
\begin{equation}
\begin{split}
 \Pi \Pi_{ss}\Pi(W) &= -2\Pi_s^2\Pi(W) = -2\Pi_s(-X_j g^{ij} \langle \Pi(\nabla_{X_i} X_s), W\rangle) 
 \\&= -2g^{ij}\langle W,\Pi(\nabla_{X_i} X_s)\rangle \Pi(\nabla_{X_j}X_s). 
\end{split}
\end{equation}
 
 Similarly, differentiate $\Pi(X_i) =0$ twice to get $\Pi_{ss}(X_i) + 2\Pi_s(X_{si}) + \Pi(X_{ssi})=0$. Together with (\ref{eq:dPi}) this immediately yields (\ref{eq:d2PiTan}). 
   
 Finally, since $\Pi$ is a symmetric operator, so too are its derivatives, and for any $i$ we have $\langle X_i , \Pi_{ss}\Pi(W) \rangle = \langle \Pi_{ss}(X_i), \Pi(W)\rangle$. This gives (\ref{eq:d2PiNT}) as only the first two terms in (\ref{eq:d2PiTan}) are normal.  \end{proof}

Henceforth we assume at $s=0$ that $X_s = V$, $V^\top=0$, and $X_{ss}=0$. We record some relations for $V$:

\begin{lemma}
\label{lem:dV}
Let $V$ be a normal vector field on $\Sigma$, then
\begin{alignat*}{3}
&\nabla^\top_j V &&= -X_l g^{lk} \langle A_{jk}, V\rangle,\\
&\nabla^\top_i\nabla^\perp_j V &&= -X_l g^{lk} \langle A_{ik}, \nabla^\perp_j V\rangle,\\
&\nabla_i\nabla^\top_j V &&=  -X_{il} g^{lk} \langle A_{jk}, V\rangle -  X_l g^{lk} \langle A_{jk}, \nabla^\perp_iV\rangle  -X_l g^{lk} \langle \nabla_i A_{jk}, V\rangle,
\end{alignat*}

\end{lemma}
\begin{proof}
Note that
$\langle \nabla_i V,X_j\rangle=-\langle V,\nabla_i X_j\rangle=-\langle V,A_{ij}\rangle,$
which implies the first two formulae. The third follows by differentiating the above. \end{proof}

\begin{proposition}
At $s=0$, assume that $X_s = V = V^\perp$ and $X_{ss}=0$; then we have
\begin{equation}
(g_{ij})_{ss} = 2 g^{kl}A^V_{ik} A^V_{lj} + 2\langle \nabla^\perp_i V, \nabla_j^\perp V\rangle, 
\end{equation}

\begin{equation}
\label{eq:d2ginv}
(g^{ij})_{ss} = g^{i i_2} g^{jj_2}( 6 g^{kl} A^V_{i_2 k} A^V_{l j_2} - 2\langle \nabla^\perp_{i_2} V, \nabla_{j_2}^\perp V\rangle), 
\end{equation}

\begin{equation}
\label{eq:d2A}
\begin{split}
\frac{1}{2}(A_{ij})_{ss} =& - X_k \langle (\nabla^\perp\nabla^\perp V)_{ij} ,\nabla^\perp_k V\rangle + X_k A^V_{jl} \langle A_{il} , \nabla_k^\perp V\rangle \\&+  \left(\langle A_{ik},\nabla_j^\perp V\rangle +\langle A_{jk}, \nabla_i^\perp V\rangle + \langle (\nabla A)_{jki}, V\rangle  \right)\nabla_k^\perp V 
\\&  - X_k A^V_{km} \langle A_{ij}, \nabla^\perp_m V\rangle - \langle A_{ij}, \nabla^\perp_k V\rangle \nabla^\perp_k V.
\end{split}
\end{equation}

\end{proposition}
\begin{proof}
Differentiate $g_{ij} = \langle X_i,X_j\rangle$ twice, using $X_{ss}=0$ to get \[(g_{ij})_{ss} = 2\langle F_{si,} ,F_{sj}\rangle = 2 g^{kl}A^V_{ik} A^V_{lj} + 2\langle \nabla^\perp_i V, \nabla_j^\perp V\rangle.\] Here we used Lemma \ref{lem:dV} for the tangent part of $X_{si}$. Then differentiate $g^{ij} g_{jk} = \delta^i_k$ twice to find that $(g^{ij})_{ss} g_{jk} = -2(g^{ij})_s(g_{jk})_s - g^{ij} (g_{jk})_{ss}$. Together with (\ref{eq:dg}) this gives (\ref{eq:d2ginv}). 

Now as $A_{ij}=\Pi(X_{ij})$, we have $A''_{ij} =  2\Pi_s(X_{sij}) + \Pi_{ss}(X_{ij})$. By (\ref{eq:dPi}) and Lemma \ref{lem:dV},

\begin{alignat*}{3}
\Pi_s (X_{sij}) &=&& - X_k \langle \nabla_i^\perp \nabla_j V, \nabla_k^\perp V\rangle - \nabla^\perp_{\nabla_i^\top \nabla_j V} V
\\&=&& -X_k \langle \nabla_i^\perp \nabla_j^\perp V,\nabla_k^\perp V\rangle + X_k A^V_{jl} \langle A_{il} , \nabla_k^\perp V\rangle  
\\&&& +  (\langle A_{ik},\nabla_j^\perp V\rangle +\langle A_{jk}, \nabla_i^\perp V\rangle + \langle \nabla_i A_{jk}, V\rangle + A^V_{jl} \langle X_{il}^\top, X_k\rangle )\nabla_k^\perp V 
\end{alignat*}

Using Lemmas \ref{lem:d2Pi} and \ref{lem:dV} we have

\begin{alignat*}{3}
\frac{1}{2}\Pi_{ss}(X_{ij}) &={}&& \nabla^\perp_{\nabla^\top_{X_{ij}^\top}V}V + X_l g^{kl} \langle \nabla^\perp_k V, \nabla^\perp_{X_{ij}^\top} V\rangle \\ &&&+ X_l g^{kl} \langle A_{ij}, \nabla^\perp_{\nabla^\top_k V}V\rangle - g^{kl} \langle A_{ij}, \nabla^\perp_k V\rangle \nabla^\perp_l V 
\\&={}&&  A^V _{kl} \langle X_{ij}^\top, X_l\rangle  \nabla^\perp_k V + X_l g^{kl} \langle \nabla^\perp_k V, \nabla^\perp_{X_{ij}^\top} V\rangle \\ &&& - X_l g^{kl} A^V_{km} \langle A_{ij}, \nabla^\perp_m V\rangle - g^{kl} \langle A_{ij}, \nabla^\perp_k V\rangle \nabla^\perp_l V.
\end{alignat*}

Combining these gives (\ref{eq:d2A}). 

\end{proof}

\begin{proposition}
At $s=0$, assume that $X_s = V = V^\perp$ and $X_{ss}=0$; then we have
\begin{equation}
\label{eq:d2phi}
\begin{split}
\frac{1}{2}\phi_{ss}=  \frac{1}{2}\mathcal{D}^2\mathcal{\varphi}(V,V)={}& \langle \nabla_k^\perp \phi, V\rangle \nabla^\perp_k V - \langle \phi, \nabla_k^\perp V\rangle \nabla^\perp_k V 
\\&+A_{ij} A^V_{ik} A^V_{kj} - A_{ij} \langle \nabla^\perp_i V, \nabla_j^\perp V\rangle + 2A^V_{ij} (\nabla^\perp\nabla^\perp V)_{ij} 
 \\& +2\langle A_{ij}, \nabla_i^\perp V\rangle \nabla_j^\perp V 
 \\& -X_k ( A^V_{ik} \langle \phi,\nabla^\perp_i V\rangle - \langle LV, \nabla^\perp_k V\rangle) .
 \end{split}
\end{equation}

\end{proposition}
\begin{proof}
First, differentiating $\mathbf{H} = - g^{ij} A_{ij}$ twice we have \[-\mathbf{H}_{ss} = (g^{ij})_{ss} A_{ij} + 2(g^{ij})_s (A_{ij})_s  + g^{ij} (A_{ij})_{ss}. \]

Now by Proposition \ref{prop:1stvar} and (\ref{eq:d2ginv}), 
 \[(g^{ij})_{ss} A_{ij} =6A_{ij} A^V_{ik} A^V_{kj} - 2A_{ij} \langle \nabla^\perp_i V, \nabla_j^\perp V\rangle,\] 
\[(g^{ij})_s (A_{ij})_s = - 2 X_k A^V_{ij} \langle A_{ij}, \nabla^\perp_k V\rangle + 2A^V_{ij} (\nabla^\perp\nabla^\perp V)_{ij} -2A^V_{ij} A^V_{ik} A_{jk}.\]
Finally, using (\ref{eq:d2A}), and the Codazzi equation to rearrange indices, we have  
\begin{equation}
\begin{split}
\frac{1}{2}g^{ij}(A_{ij})_{ss} =& - X_k \langle \Lap^\perp V ,\nabla^\perp_k V\rangle + X_k  A^V_{ij} \langle A_{ij} , \nabla_k^\perp V\rangle \\&+  \left(2\langle A_{ik}, \nabla_i^\perp V\rangle - \langle \nabla_k^\perp \mathbf{H},  V\rangle  \right)\nabla_k^\perp V 
\\&  + X_k A^V_{ki} \langle \mathbf{H}, \nabla^\perp_i V\rangle +  \langle \mathbf{H}, \nabla^\perp_k V\rangle \nabla^\perp_k V.
\end{split}
\end{equation}

Adding these together gives \begin{equation}
\label{eq:d2H}
\begin{split}
 -\mathbf{H}_{ss} = {}& 2A_{ij} A^V_{ik} A^V_{kj} - 2A_{ij} \langle \nabla^\perp_i V, \nabla_j^\perp V\rangle + 4A^V_{ij} (\nabla^\perp\nabla^\perp V)_{ij} 
 \\& +\left(4\langle A_{ik}, \nabla_i^\perp V\rangle - 2\langle \nabla_k^\perp \mathbf{H},  V\rangle +2 \langle \mathbf{H}, \nabla^\perp_k V\rangle  \right)\nabla_k^\perp V 
 \\&  + X_k\left( - 2A^V_{ij}  \langle A_{ij}, \nabla^\perp_k V\rangle -2\langle \Lap^\perp V ,\nabla^\perp_k V\rangle  +2A^V_{ki} \langle \mathbf{H}, \nabla^\perp_i V\rangle \right) 
 .
 \end{split}
\end{equation}

Differentiating $X^\perp = \Pi(X)$ twice and using $X_{ss}=0$ we have $(X^\perp)_{ss} = \Pi_{ss}(X) + 2\Pi_s(X_s)$. Then by Lemmas \ref{lem:d2Pi} and \ref{lem:dV}, we have
\[\begin{split}
 \frac{1}{2}\Pi_{ss}(X)={}&  \nabla^\perp_{\nabla^\top_{X^\top}V}V + X_k \langle \nabla^\perp_k V, \nabla^\perp_{X^\top} V\rangle  + X_k \langle X^\perp, \nabla^\perp_{\nabla^\top_{k} V} V\rangle - \langle X^\perp, \nabla^\perp_k V\rangle \nabla^\perp_k V
 \\=& - A^V_{jk} \langle X, X_j\rangle \nabla^\perp_k V + X_k \langle \nabla^\perp_k V, \nabla^\perp_{X^\top} V\rangle - X_k A^V_{jk} \langle X^\perp, \nabla^\perp_j V\rangle- \langle X^\perp, \nabla^\perp_k V\rangle \nabla^\perp_k V. 
\end{split}\]
 By Proposition \ref{prop:1stvar} we have
$\Pi_s(X_s) = -X_k \langle \nabla^\perp_k V, V\rangle.$ Adding this then gives
\begin{equation}
\label{eq:d2X}
\begin{split}
 \frac{1}{2}(x^\perp)_{ss} =&  - A^V_{jk} \langle X, X_j\rangle \nabla^\perp_k V - \langle X^\perp, \nabla^\perp_k V\rangle \nabla^\perp_k V
  \\& + X_k \langle \nabla^\perp_k V, \nabla^\perp_{X^\top} V\rangle - X_k A^V_{jk} \langle X^\perp, \nabla^\perp_j V\rangle-X_k \langle \nabla^\perp_k V, V\rangle
 .
 \end{split}
\end{equation}

Finally, differentiating $X^\perp=\Pi(X) = X - g^{ij} \langle X, X_i\rangle X_j$, we find $\nabla_k^\perp X^\perp = - g^{ij}\langle X, X_i\rangle A_{jk}$. Therefore $\nabla_k^\perp \phi = -\frac{1}{2}\langle X,X_j\rangle A_{jk}  -\nabla_k^\perp \mathbf{H}$. Adding (\ref{eq:d2H}) and (\ref{eq:d2X}) then gives the result. 
\end{proof}

\begin{corollary}
\label{cor:d2phiN}
Suppose that $\Sigma_0$ is a shrinker, that is, at $s=0$ we have $\phi \equiv 0$; then we have \begin{equation}
\begin{split}
\frac{1}{2}\phi_{ss} = \frac{1}{2}\mathcal{D}^2\mathcal{\varphi}(V,V) ={}&   A_{ij}A^V_{ik} A^V_{kj}  - A_{ij}\langle \nabla^\perp_i V, \nabla^\perp_j V\rangle \\&+ 2A^V_{ij} (\nabla^\perp\nabla^\perp V)_{ij} 
 +2\langle A_{ij}, \nabla_i^\perp V\rangle \nabla_j^\perp V 
 \\& - X_k \langle LV, \nabla^\perp_k V\rangle. 
\end{split}
\end{equation}
\end{corollary}

\begin{remark}
Ultimately, only the normal part of $\mathcal{D}^2\mathcal{\varphi}(V,V)$ is geometrically significant; it would have been enough to take the normal projection in the above second variation formula and in all later analysis. 
\end{remark}

\section{Analysis on a cylinder}
\label{sec:analysis}

In this section we specialise to generalised shrinking cylinders $\Gamma = \mathring{\Gamma}^k\times \mathbb{R}^{n-k}\subset \mathbb{R}^N$ satisfying (A1-A2). We describe the constraint posed by the first variation, or equivalently the space of Jacobi fields, and then show that the second variation completes the formal obstruction, in particular, it has a definite sign on Jacobi variations. 

The normal space has the orthogonal decomposition $N_{(p,y)}\Gamma = N_p\mathring{\Gamma} \oplus \spa\{\pr_{\alpha}\}$. Here $N\mathring{\Gamma}$ is the normal bundle of $\mathring{\Gamma}^k$ in $\mathbb{R}^{k+m}\hookrightarrow \mathbb{R}^N$. We correspondingly decompose $V \in N\Gamma$ as $V = \mathring{V} + \sum_\alpha v^\alpha \pr_{z_\alpha}$, where $\mathring{V}$ is the projection of $V\in N\Gamma$ to $N\mathring{\Gamma}$. Similarly, the second fundamental form on $\Gamma$ is $A(Y,Z) = \mathring{A}(\mathring{Y}, \mathring{Z})$, where $\mathring{A}=A_{\mathring{\Gamma}}$. In particular $\mathbf{H}(p,y)=\mathring{\mathbf{H}}(p)$.

It will be convenient to factor the Gaussian weight on $\Gamma$ as $\rho = \mathring{\rho}\bar{\rho}$, where the Gaussians on $\mathring{\Gamma}$, $\mathbb{R}^{n-k}$ may be written explicitly as  $\mathring{\rho} = (4\pi)^{-k/2} \e^{-|\mathring{x}|^2/4}$, $\bar{\rho} = (4\pi)^{-(n-k)/2} \e^{-|y|^2/4}$. 

\subsection{First variation of $\phi$}
\label{sec:analysis1}

Since $\Gamma$ is a shrinker, the first variation of $\phi$ is just given by the Jacobi operator $\mathcal{D\phi}(V)=LV$. The goal for this subsection is to describe the Jacobi fields on $\Gamma$, that is, the space $\mathcal{K}:= \ker L = \{ V\in W^{2,2}(N\Gamma) | LV=0\}$. Let $\pi_\mathcal{K}$ be the $L^2$ projection to $\mathcal{K}$ and $\pi_{\mathcal{K}^\perp}$ the projection to its orthocomplement $\mathcal{K}^\perp$ in $L^2$. 

We first record some results from functional analysis:

\begin{lemma}
\label{lem:spectral-decomp}
Let $T$ be either operator $\mathcal{L}$ (acting on functions or normal fields) or $L$ on $\Gamma \in \mathcal{C}_n(\mathring{\Gamma})$. Then the operator $T$ is symmetric on $W^{2,2}$, the space $W^{1,2}$ embeds compactly into $L^2$, and $T$ has discrete spectrum with finite multiplicity on $W^{2,2}$ and a complete basis of smooth $L^2$-orthonormal eigenfunctions. 
\end{lemma}
\begin{proof}
As in \cite[Lemma 3.2]{CM15}, the lemma follows from integration by parts, \cite{BE85} and \cite{Gri09}, noting that $\Gamma$ has a finite lower bound for its Bakry-\'{E}mery-Ricci curvature, and finite weighted volume. The results for $L$ follow since $A$ is also uniformly bounded on $\Gamma$. 
\end{proof}

We will need the following elliptic estimate. For convenience, define $|V|_{m} =\sum_{j\leq m} |\nabla^j V|$; then we have

\begin{lemma}
\label{lem:elliptic-est}
Fix a shrinker $\Gamma = \mathring{\Gamma}\times \mathbb{R}^{n-k}$. Then the following hold:
\begin{enumerate}[(a)]
\item Given $m\in \mathbb{N}$, there exists $C$ so that if $\int_\Gamma \langle x\rangle^{2m} |V|^2_2\rho <\infty$, then
\begin{equation}
\label{eq:elliptic-est}
\int_\Gamma \langle x\rangle^{2m} |V|_2^2 \rho  \leq C  \|V\|_{L^2}^2  + C\int_\Gamma \langle x\rangle ^{2m} |LV|^2\rho.  
\end{equation}
\item There exists $C_0$ so that if $V\in W^{2,2}$ then $\|\pi_{\mathcal{K}^\perp}(V)\|_{W^{2,2}} \leq C_0 \|LV\|_{L^2}$. 
\end{enumerate}
\end{lemma}

\begin{proof}
Note that $L$ only differs from $\mathcal{L}$ by (zeroth order) curvature terms, which are uniformly bounded on the cylinder $\Gamma$. Therefore, we may freely use that $|\mathcal{L}V- LV| \leq C|V|$. 

First, integrate by parts to find that \[ \int_\Gamma |x|^{2m} |\nabla^\perp V|^2 \rho = - \int_\Gamma | x|^{2m} \langle V, \mathcal{L}V\rangle \rho - 2m \int_\Gamma |x|^{2m-2} \langle V, \nabla^\perp_{x^\top} V\rangle \rho.\] Using the absorbing inequality on the last term, it follows that \begin{equation}\label{eq:elliptic-est1}\begin{split}\int_\Gamma \langle x\rangle^{2m} |\nabla^\perp V|^2 \rho  &\leq  2 \int_\Gamma \langle x\rangle ^{2m} |\mathcal{L}V|^2\rho + C \int_\Gamma \langle x\rangle^{2m} |V| ^2 \rho\\&\leq 4 \int_\Gamma \langle x\rangle ^{2m} |LV|^2\rho + C' \int_\Gamma \langle x\rangle^{2m} |V| ^2 \rho.\end{split}\end{equation}

We proceed to estimate $\int_\Gamma \langle x\rangle^{2m} |\nabla^\perp\nabla^\perp V|^2\rho$. Since $\Gamma = \mathring{\Gamma}\times \mathbb{R}^{n-k}$, the Ricci curvature of the normal bundle $N\Gamma$ is given by $\Ric_\Gamma^\perp(Y,Z) V= \Ric_{\mathring{\Gamma}}^\perp ( \mathring{Y},\mathring{Z}) \mathring{V}$, where $\mathring{Y}(p),\mathring{Z}(p)$ are the projections to $T_p\mathring{\Gamma}$ and $\mathring{V}(p)$ is the projection to $N_p\mathring{\Gamma}$. In particular $\Ric^\perp$ and $\nabla^\perp \Ric^\perp$ are uniformly bounded, and the drift Bochner inequality \begin{equation}\label{eq:bochner}|\nabla^\perp\nabla^\perp V|^2\leq \frac{1}{2}\mathcal{L}|\nabla^\perp V|^2 - \langle \nabla^\perp \mathcal{L} V,\nabla^\perp V\rangle + K|\nabla^\perp V|^2 + K|V|^2\end{equation} holds for some $K<\infty$. On the other hand, integrating by parts we have \begin{equation}\label{eq:elliptic-est-2a}  \int_\Gamma |x|^{2m}\mathcal{L}|\nabla^\perp V|^2 \rho = - 2m \int_\Gamma |x|^{2m-2} \langle x^\top , \nabla |\nabla^\perp V|^2 \rangle \rho \leq C m \int_\Gamma \langle x \rangle^{2m-1} |\nabla^\perp V||\nabla^\perp\nabla^\perp V| \rho.\end{equation} 
Similarly, by direct calculation we have $\mathcal{L}|x|^{2m} \leq C(m)\langle x\rangle^{2m}$, so
\begin{equation}
\label{eq:elliptic-est-2b}
\begin{split}
 -\int_\Gamma |x|^{2m}\langle \nabla^\perp \mathcal{L} V,\nabla^\perp V\rangle\rho &= \int_\Gamma \langle \mathcal{L} V, \mathcal{L}(|x|^{2m}V)\rangle\rho 
 \\&\leq C\int_\Gamma |x|^{2m} |\mathcal{L}V|^2 + C\int_\Gamma \langle x\rangle^{2m-2}|\nabla^\perp V|^2\rho + C\int_\Gamma \langle x\rangle^{2m} |V|^2 \rho .
 \end{split}
 \end{equation}

Integrating (\ref{eq:bochner}), using (\ref{eq:elliptic-est-2b}) and the absorbing inequality on (\ref{eq:elliptic-est-2a}), it follows that 
\[ \int_\Gamma \langle x\rangle^{2m} |\nabla^\perp\nabla^\perp V|^2 \rho \leq C\int_\Gamma \langle x\rangle ^{2m} |\mathcal{L}V|^2\rho + C\int_\Gamma \langle x\rangle^{2m} |\nabla^\perp V|^2\rho+C\int_\Gamma \langle x\rangle^{2m} |V|^2\rho .\] 
The second last term was already estimated by (\ref{eq:elliptic-est1}), so this yields the estimate 
\begin{equation}
\label{eq:elliptic-est2}
\begin{split}
\int_\Gamma \langle x\rangle^{2m} |V|_2^2 \rho  &\leq C   \int_\Gamma \langle x\rangle^{2m} |V| ^2 \rho + C\int_\Gamma \langle x\rangle ^{2m} |\mathcal{L}V|^2\rho
\\&\leq C'  \int_\Gamma \langle x\rangle^{2m} |V| ^2 \rho + C\int_\Gamma \langle x\rangle ^{2m} |LV|^2\rho.
\end{split}
\end{equation}

This already establishes (\ref{eq:elliptic-est}) for $m=0$. We proceed by induction on $m$, so suppose that 
\begin{equation} 
\label{eq:elliptic-est3}
\int_\Gamma \langle x\rangle^{2m-2} |V|_2^2 \rho  \leq C  \|V\|_{L^2}^2  + C\int_\Gamma \langle x\rangle ^{2m-2} |LV|^2\rho. 
\end{equation} 
On the other hand, it follows from Lemma \ref{lem:poincare} that
\begin{equation}
\int_\Gamma \langle x\rangle^{2m} |V| ^2 \rho \leq \int_\Gamma \langle x\rangle^{2m-2} |V| ^2 \rho + \int_\Gamma \langle x\rangle^{2m-2} |\nabla^\perp V| ^2 \rho.
\end{equation}
Using (\ref{eq:elliptic-est1}) for the last term and then applying the inductive hypothesis (\ref{eq:elliptic-est3}) gives that 
\begin{equation}
\int_\Gamma \langle x\rangle^{2m} |V| ^2 \rho \leq C  \|V\|_{L^2}^2  + C\int_\Gamma \langle x\rangle ^{2m} |LV|^2\rho. 
\end{equation} 
Substituting this into (\ref{eq:elliptic-est2}) completes the induction and establishes (\ref{eq:elliptic-est}) for all $m$. 

For item (b), Lemma \ref{lem:spectral-decomp} for $L$ implies that $\|\pi_{\mathcal{K}^\perp}(V)\|_{L^2} \leq \Lambda \|LV\|_{L^2}$ for some $\Lambda$. Applying the $m=0$ case of (a) then gives \[\|\pi_{\mathcal{K}^\perp}(V)\|_{W^{2,2}} \leq C_0(\|\pi_{\mathcal{K}^\perp}(V)\|_{L^2} + \|LV\|_{L^2}) \leq C_0 (1+\Lambda) \|LV\|_{L^2}\]
\end{proof}

Recall that, acting on functions, the ($L^2$) spectrum of $\bar{\mathcal{L}}$ on Euclidean space consists of the nonnegative half-integers; moreover $ \ker(\bar{\mathcal{L}} +\frac{1}{2}) = \spa \{y_i\}$ and $  \ker(\bar{\mathcal{L}}+1) = \spa\{ y_iy_j-2\delta_{ij}\}$. 

\begin{proposition}
\label{prop:jacobi-fields}
The space $\mathcal{K}$ of ($W^{2,2}$) Jacobi fields on $\Gamma = \mathring{\Gamma} \times \mathbb{R}^{n-k}$ is spanned by elements of the following forms:
\begin{enumerate}
\item $\mathring{V}(p) f(y)$, where $\mathring{V}\in \ker (\mathring{L}-\frac{b}{2})$, $f \in \ker(\bar{\mathcal{L}}+\frac{b}{2})$, $b\in \mathbb{N}$; 
 \item $y_i \pr_{z_\alpha}$; 
 \item $v(p)\pr_{z_\alpha}$, where $v\in\ker(\mathring{\mathcal{L}} + \frac{1}{2})$.
 \end{enumerate}
\end{proposition}
\begin{proof}
Using the cylindrical structure of $\Gamma$, the Jacobi operator acts on normal vector fields $V = \mathring{V} + \sum_\alpha v^\alpha \pr_{z_\alpha}$ by $L\mathring{V} = \mathring{L} \mathring{V} + \bar{\mathcal{L}} \mathring{V}$, where $\bar{\mathcal{L}}$ is the drift Laplacian on the linear factor $\mathbb{R}^{n-k}$, and $L(v^\alpha \pr_{z_\alpha}) = (\mathring{\mathcal{L}}+\bar{\mathcal{L}}+ \frac{1}{2})(v^\alpha) \pr_{z_\alpha}$. 

In particular the operator $L$ respects the decomposition of $N\Gamma$, and we proceed by separation of variables, noting that $\mathring{\mathcal{L}}$ and $\bar{\mathcal{L}}$ commute. 

First we consider the action on functions: Let $\{\psi_j\}$ be an $L^2$-orthonormal eigenbasis of $\mathring{\mathcal{L}}$ on $\mathring{\Gamma}$, so that $\mathring{\mathcal{L}}\psi_j = -\mu_j \psi_j$, $\mu_j \geq 0$. Suppose that $u$ is a $W^{2,2}$ function on $\Gamma$ with $(\mathring{\mathcal{L}} + \bar{\mathcal{L}} + \frac{1}{2})u=0$. Then $u$ is smooth by (local) elliptic regularity, and by fixing $y$ we may write $u(p,y) = \sum_j \psi_j(p) u^j(y)$, where $u^j(y) = \int_{\mathring{\Gamma}} u(\cdot,y) \psi_j \mathring{\rho}$. This series is square summable and hence converges in $W^{2,2}$; in particular each $u^j$ is $W^{2,2}$ on $\mathbb{R}^{n-k}$, and we can calculate \[0=(\mathring{\mathcal{L}} + \bar{\mathcal{L}} + \frac{1}{2})u = \sum_j \psi_j(\bar{\mathcal{L}}-\mu_j +\frac{1}{2})u^j.\] By the $L^2$-orthogonality, each $u_j$ must satisfy $(\bar{\mathcal{L}}-\mu_j +\frac{1}{2})u^j=0$. So by the characterisation of the spectrum of $\bar{\mathcal{L}}$, we must have $-\mu_j + \frac{1}{2} \in \frac{1}{2}\mathbb{N}$. This is only possible if:
\begin{itemize}
\item $\mu_j = 0$, in which case $\psi_j$ is constant and $u^j \in \ker(\bar{\mathcal{L}} + \frac{1}{2}) = \langle y_i\rangle$;
\item $\mu_j = \frac{1}{2}$, in which case $u^j$ must be constant. 
\end{itemize}

The action on normal fields orthogonal to the linear directions is similar: Let $\{\mathring{V}_j\}$ be an $L^2$-orthonormal eigenbasis of $\mathring{L}$ on $N\mathring{\Gamma}$, so that $\mathring{L}\mathring{V}_j = -\lambda_j \mathring{V}_j$. Suppose that $\mathring{V}$ is a $W^{2,2}$ normal vector field on $\Gamma$ such that $\mathring{V}(p,y) \in N_p\mathring{\Gamma}$, and $L\mathring{V} = \mathring{L} \mathring{V} + \bar{\mathcal{L}} \mathring{V}=0$. Then $\mathring{V}$ is smooth, and by fixing $y$ we may write $\mathring{V}(p,y) = \sum_j \mathring{V}_j(p) v^j(y)$, where $v^j(y) = \int_{\mathring{\Gamma}} \langle \mathring{V}(\cdot,y), \mathring{V}_j\rangle \mathring{\rho}$. Again this series is square summable and hence converges in $W^{2,2}$, so each $v^j$ is a $W^{2,2}$ function on $\mathbb{R}^{n-k}$, and we can calculate \[0=Lu = \sum_j (\bar{\mathcal{L}}(v^j\mathring{V}_j) -\lambda_j \mathring{V}_j )=\sum_j \mathring{V}_j (\bar{\mathcal{L}} -\lambda_j ) v^j .\] Then each $v_j$ satisfies $(\bar{\mathcal{L}}-\lambda_j)v^j=0$. Using the characterisation of $\spec \bar{\mathcal{L}}$ again completes the proof.

\end{proof}

It is also convenient to define $\mathcal{K}_1 := \mathcal{K} \cap \mathcal{K}_0^\perp$ to be the $L^2$-orthocomplement of $\mathcal{K}_0$ in $\mathcal{K}$; that is, the space of Jacobi fields orthogonal to rotations. 

\begin{corollary}
\label{cor:jacobimodrot}
Let $\Gamma =\mathring{\Gamma} \times \mathbb{R}^{n-k}$ where $\mathring{\Gamma}$ is a closed shrinker satisfying (A1-A2). Then $\mathcal{K}_1$ is spanned by the normal fields $\{(y_iy_j -2\delta_{ij})\mathbf{H}\}$. 
\end{corollary}
\begin{proof}
Recall from Lemma \ref{lem:rot-fields} that $\mathcal{K}_0$ is spanned by normal fields of the following types: 
\begin{enumerate}[(i)]
\item $\mathring{x}_i\pr_{\mathring{x}_j}^\perp - \mathring{x}_j\pr_{\mathring{x}_i}^\perp$; 
\item $y_j \pr_{\mathring{x}_i}^\perp$; 
\item $\mathring{x}_i \pr_{z_\alpha}$; and $y_j \pr_{z_\alpha}$. 
\end{enumerate}

We now compare each case (1-3) of Proposition \ref{prop:jacobi-fields} to the above types (i-iii). Case (2) immediately reduces to type (iii). By (A1), the only candidate functions $v$ in case (3) are the coordinate functions $\mathring{x}_i$ on $\mathring{\Gamma}\subset \mathbb{R}^{k+m} $, so these are also contained in type (iii). By (A2), we have three possibilities for $b$ in case (1):

\begin{itemize}
\item $b=0$, in which case $f(y)$ is constant and $\mathring{V}$ is precisely of type (i);
\item $b=1$; in which case $f(y) \in \langle y_j\rangle$ and $\mathring{V}\in \langle \pr_{x_i}^\top\rangle$, which is of type (ii) above; 
\item $b=2$; in which case $f(y) \in \langle y_iy_j-2\delta_{ij}\rangle$ and $\mathring{V} = \mathbf{H}$. 
\end{itemize}
Thus the $b=2$ subcase is the only case not covered by rotations.

To check that $(y_iy_j -2\delta_{ij})\mathbf{H}$ is $L^2$-orthogonal to each type in $\mathcal{K}_0$, note that type (i) is orthogonal since the rotation $\mathring{x}_i\pr_{\mathring{x}_j}^\perp - \mathring{x}_j\pr_{\mathring{x}_i}^\perp$ and $\mathbf{H}=\mathring{\mathbf{H}}$ are in distinct eigenspaces of $\mathring{L}$. Type (ii) is orthogonal since it has odd degree in $y$, and type (iii) is pointwise orthogonal. This completes the proof. 
\end{proof}

\begin{corollary}
\label{cor:jacobi-est}
Let $\Gamma =\mathring{\Gamma} \times \mathbb{R}^{n-k}$ where $\mathring{\Gamma}$ is a closed shrinker satisfying (A1-A2). Let $r_0=\diam(\mathring{\Gamma})+1$. There exists $C$ depending on $N,\mathring{\Gamma}$ so that if $J \in \mathcal{K}$, then $|J| \leq C\langle x\rangle^2 \|J\|_{L^2(B_{r_0})}$, $|\nabla J| + |\nabla^2 J| \leq C\langle x\rangle^2 \|J\|_{L^2(B_{r_0})}$ and $|\nabla^2 J ( \cdot,\pr_{y_i})|  \leq C\langle x\rangle\|J\|_{L^2(B_{r_0})} $. 
\end{corollary}
\begin{proof}
By Corollary \ref{cor:jacobimodrot} and as in Corollary \ref{cor:rotation-est}, we need only check the growth of the normal fields $(y_iy_j - 2\delta_{ij})\mathring{\mathbf{H}}$. These clearly grow quadratically in $y$ (hence $\langle x\rangle$); for the last estimate note that after a derivative in a $y$-direction, the growth is at most linear.  
\end{proof}

\subsection{Second variation of $\phi$}
\label{sec:analysis2}

The main goal of this subsection is to use Corollary \ref{cor:d2phiN} to calculate the $L^2$-projection of the second variation of $\phi$ to the quadratic Jacobi fields $(y_j^2-2)\mathbf{H}$. 
First, recall that the definition of simply non-integrable involved the quantity

\[B_1(\mathring{\Gamma}) = \int_{\mathring{\Gamma}} \left( A^\mathbf{H}_{ij} A^\mathbf{H}_{il} A^\mathbf{H}_{lj} - 3 A^\mathbf{H}_{ij} \langle \nabla^\perp_i \mathbf{H}, \nabla_j^\perp \mathbf{H}\rangle  \right)\rho.  \]

The slightly more complictated quantity that falls out of the second variation is instead

\[\begin{split}B_2(\mathring{\Gamma})  ={}& \int_{\mathring{\Gamma}} \left( A^\mathbf{H}_{ij} A^\mathbf{H}_{ik} A^\mathbf{H}_{kj} - A^\mathbf{H}_{ij} \langle \nabla^\perp_i \mathbf{H}, \nabla_j^\perp \mathbf{H}\rangle\right)\rho \\ &+ \int_{\mathring{\Gamma}}\left( 2A^\mathbf{H}_{ij} \langle(\nabla^\perp\nabla^\perp \mathbf{H})_{ij},\mathbf{H}\rangle + 2 \langle A_{ij},\nabla_i^\perp \mathbf{H}\rangle \langle \nabla_j^\perp \mathbf{H},\mathbf{H} \rangle \right)\rho.\end{split}\]

A quick integration by parts shows that $B_1$ and $B_2$ coincide for shrinkers:

\begin{lemma}
Let $\mathring{\Gamma}$ be a compact submanifold satisfying $\nabla^\perp \phi=0$. Then $B_1(\mathring{\Gamma}) = B_2(\mathring{\Gamma})$. 
\end{lemma}
\begin{proof}
Integrating the third term by parts, we find that

\begin{equation}
\begin{split}
\int_{\mathring{\Gamma}} \langle A_{ij},\mathbf{H}\rangle \langle(\nabla^\perp\nabla^\perp \mathbf{H})_{ij},\mathbf{H}\rangle \rho =& -\int_{\mathring{\Gamma}} \left(\langle A_{ij} , \nabla_i^\perp \mathbf{H}\rangle \langle \nabla_j^\perp \mathbf{H}, \mathbf{H}\rangle + \langle A_{ij}, \mathbf{H}\rangle \langle \nabla_j^\perp \mathbf{H}, \nabla_i^\perp \mathbf{H}\rangle\right)\rho \\& - \int_{\mathring{\Gamma}} \langle (\nabla^\perp A)_{ij,i} ,\mathbf{H}\rangle \langle \nabla^\perp_j \mathbf{H}, \mathbf{H}\rangle \rho \\&+ \frac{1}{2}\int_{\mathring{\Gamma}}  \langle A_{ij}, \mathbf{H}\rangle \langle \nabla_j^\perp \mathbf{H}, \mathbf{H}\rangle \langle x^\top , X_i\rangle\rho.
\end{split}
\end{equation}
The last two terms cancel as the Codazzi equation implies $(\nabla^\perp A)_{ij,j} = -\nabla_j^\perp \mathbf{H}$, and $\nabla^\perp \phi=0$ implies $A_{ij} \langle x^\top,X_i\rangle = -2\nabla_j^\perp \mathbf{H}$. Collecting the remaining terms completes the proof. 
\end{proof}

We also record some elementary integration results:

\begin{lemma}
\label{lem:int-poly}
There holds
\begin{enumerate}
\item $\frac{1}{\sqrt{4\pi}} \int_\mathbb{R} y^m \e^{-y^2/4} \d y=0$ for any odd $m\in \mathbb{N}$;
\item $\frac{1}{\sqrt{4\pi}} \int_\mathbb{R} (y^2-2) \e^{-y^2/4} \d y=0$; 
\item $\frac{1}{\sqrt{4\pi}} \int_\mathbb{R} (y^2-2)^2 \e^{-y^2/4} \d y = 8$; 
\item $\frac{1}{\sqrt{4\pi}} \int_\mathbb{R} (y^2-2)^3 \e^{-y^2/4} \d y = 64$.
\end{enumerate}
\end{lemma}

We now prove the formal obstruction to be used later in the Taylor expansion:

 \begin{proposition}
 \label{prop:d2phiEst}
Let $\Gamma =\mathring{\Gamma} \times \mathbb{R}^{n-k}$ where $\mathring{\Gamma}$ is a closed shrinker satisfying (A1-A2), and $B_1(\mathring{\Gamma})\neq 0$. There exists $\delta>0$ such that for any $U\in \mathcal{K}_1$, we have \[\|\pi_\mathcal{K}(\mathcal{D}^2\mathcal{\varphi}(U,U))\|_{L^2} \geq \delta \|U\|_{L^2}^2.\] 
 \end{proposition}
 \begin{proof}
By Corollary \ref{cor:jacobimodrot}, we may write any $U\in \mathcal{K}_1$ as $U = u(y) \mathbf{H}$, where \[u(y) = \sum_{ij} a_{ij}(y_iy_j-2\delta_{ij}).\] (Here it is convenient to sum over all $i,j$ and take $a_{ij}=a_{ji}$.) By Lemma \ref{lem:int-poly}(1-3), we immediately have \[\|U\|_{L^2}^2 = 8 \lambda(\mathring{\Gamma}) \sum_{ij} a_{ij}^2.\] Now since $u$ depends only on $y$, if $e_i,e_j$ are tangent to $\mathring{\Gamma}$ we can calculate $\nabla_i^\perp U = u\nabla_i^\perp \mathbf{H}$ and $(\nabla^\perp\nabla^\perp U) (e_i,e_j) = u (\nabla^\perp\nabla^\perp \mathbf{H})(e_i,e_j)$. On the other hand, $A_{ij}$ is given by $\mathring{A}_{ij}$ if $e_i,e_j$ are tangent to $\mathring{\Gamma}$, and zero otherwise, so substituting into the second variation formula Corollary \ref{cor:d2phiN}, we immediately have

\begin{equation}\label{eq:D2phiUUH} \langle \mathcal{D}^2\mathcal{\varphi}(U,U), v\mathbf{H} \rangle_{L^2} = 2B_2(\mathring{\Gamma}) \int_{\mathbb{R}^{n-k}} u^2 v \rho_{n-k}(y) \d y. \end{equation}

By the previous lemma, since $\mathring{\Gamma}$ is a shrinker we have $B_1(\mathring{\Gamma}) = B_2(\mathring{\Gamma})$. 

It follows that $\langle \mathcal{D}^2\mathcal{\varphi}(U,U), (y_b^2-2)\mathbf{H} \rangle_{L^2}= 2B_2(\mathring{\Gamma}) \sum_{ijmlb} a_{ij} a_{ml} I_{ijml;b}$, where \[ I_{ijml;b}= (4\pi)^{-\frac{n-k}{2}} \int_{\mathbb{R}^{n-k}}  (y_iy_j-2\delta_{ij})(y_m y_l -2\delta_{ml}) (y_b^2-2) \e^{-|y|^2/4}\d y.\]

For fixed $b$, one may observe using Lemma \ref{lem:int-poly}(1-2) that $I_{ijmlb}$ vanishes unless either:
\begin{itemize}
\item $i=j=m=l=b$, whence by Lemma \ref{lem:int-poly}(4), \[I_{bbbb;b} = \frac{1}{\sqrt{4\pi}} \int_\mathbb{R} (y_b^2-2)^3 \e^{-y_b^2/4} \d y_b = 64 ;\]
\item exactly one of $i,j$ and exactly one of $m,l$ is equal to $b$, and the other two indices are equal, whence by Lemma \ref{lem:int-poly}(2-3)
\[
\begin{split}I_{ibib;b} &= \left( \frac{1}{\sqrt{4\pi}} \int_\mathbb{R} y_i^2 \e^{-y_i^2/4} \d y_i\right) \left(\frac{1}{\sqrt{4\pi}} \int_\mathbb{R} y_b^2(y_b^2-2) \e^{-y_b^2/4} \d y_b \right)
\\&= \left( \frac{1}{\sqrt{4\pi}} \int_\mathbb{R} 2 \e^{-y_i^2/4} \d y_i\right) \left(\frac{1}{\sqrt{4\pi}} \int_\mathbb{R} (y_b^2-2)^2 \e^{-y_b^2/4} \d y_b \right)
\\&= 16,
\end{split}\]
where $i\neq b$. 
\end{itemize}

As there are 4 permutations of the latter case (for fixed $b$), it follows that 
\begin{equation} 
\label{eq:proj-K'}
\langle \mathcal{D}^2\mathcal{\varphi}(U,U), (y_b^2-2)\mathbf{H} \rangle_{L^2} = 128 B_1(\mathring{\Gamma}) \sum_i a_{ib}^2.
\end{equation}

Also, by Lemma \ref{lem:int-poly}(3), we have

\begin{equation}
\label{eq:proj-K'-norm}
 \|(y_j^2-2)\mathbf{H}\|_{L^2}^2 = 8\int_{\mathring{\Gamma}} |\mathbf{H}|^2 \mathring{\rho}
 \end{equation}
 
 Of course, $\mathcal{W}(\mathring{\Gamma}):= \int_{\mathring{\Gamma}} |\mathbf{H}|^2 \mathring{\rho}>0$ as there are no closed minimal submanifolds in $\mathbb{R}^{n-k}$. 
 
 Let $\mathcal{K}' =\spa \{ (y_j^2 -2)\mathbf{H}\}$. Then by (\ref{eq:proj-K'}), (\ref{eq:proj-K'-norm}) and Cauchy-Schwarz, 

\[\begin{split}
\|\pi_{\mathcal{K}'}(\mathcal{D}^2\mathcal{\varphi}(U,U))\|_{L^2} &= \frac{128|B_1(\mathring{\Gamma})|}{\sqrt{8\mathcal{W}(\mathring{\Gamma})}}\left(\sum_i \left( \sum_j a_{ij}^2\right)^2 \right)^\frac{1}{2} 
\\&\geq \frac{128|B_1(\mathring{\Gamma})|}{\sqrt{8(n-k)\mathcal{W}(\mathring{\Gamma})}} \sum_i \sum_j a_{ij}^2
\\&=\frac{16|B_1(\mathring{\Gamma})|}{\lambda(\mathring{\Gamma})\sqrt{8(n-k)\mathcal{W}(\mathring{\Gamma})} } \|U\|_{L^2}^2. 
\end{split}\]

This completes the proof as $\mathcal{K}' \subset \mathcal{K}_1\subset \mathcal{K}$. 

 \end{proof}

\section{Estimates for entire graphs}
\label{sec:entire}

In this section we describe how Taylor expansion of the shrinker quantity $\phi$ can be used to prove estimates for entire graphs $U$. The first order expansion will bound the distance to a Jacobi field, showing that the non-Jacobi part must be higher order. Expansion at the second order obstruction will then control the Jacobi part, modulo rotations. 

Throughout this section, $\Gamma = \mathring{\Gamma}^k\times \mathbb{R}^{n-k}$, where $\mathring{\Gamma}$ is a smooth closed shrinker. We work with Sobolev spaces, so a major difficulty is that Taylor expansion leads to higher powers of $U$ depending on the order of expansion. The noncompactness of $\Gamma$ also introduces polynomial factors which complicate the analysis. 

Recall we set $|V|_{m} =\sum_{j\leq m} |\nabla^j V|$ so that $\|V\|_{W^{m,q}}^q = \int_\Gamma |V|_m^q \rho$. 

\begin{lemma}
\label{lem:phi-smooth}
For any $m$ there exists $C,\delta>0$ and $\varphi$ such that if $\|U\|_{C^2} <\delta$ is any vector field on $\Gamma=\mathring{\Gamma}\times \mathbb{R}^{n-k}$, then $\phi_U = \varphi(p, U,\nabla U, \nabla^2 U)$, where $\|\varphi\|_{C^m} \leq C\langle x\rangle$. 

In fact, $\phi_U= \varphi_1(p, U, \nabla U) + \varphi_2(p, U, \nabla U, \nabla^2 U)$, where $\|\varphi_1\|_{C^m} \leq C\langle x\rangle$ and $\|\varphi_2\|_{C^m}\leq C$. 
\end{lemma}
\begin{proof}
As in \cite[Lemma 5.30]{CM19} (see also \cite[Appendix A]{CM15}), each quantity in the definition of $\phi$ may be written as a smooth function of $(p,U,\nabla U,\nabla^2 U)$. For the first (naive) estimate, note that the position vector $x(p)$ enters the explicit form of $\phi$ at most linearly. 

For the last assertion, we need slightly more care: By definition $\phi_U = \frac{1}{2}\Pi_U(x_U) - \mathbf{H}_U$. It is standard that the mean curvature $\mathbf{H}_U$ is of the form $\varphi_2$,  and by translation invariance the dependence on $p=(\mathring{p},y) \in \Gamma =\mathring{\Gamma}\times \mathbb{R}^{n-k}$ is in fact only a dependence on $\mathring{p}$. 

On the other hand, the normal projection $\Pi_U$ depends only on first order data of $U$, that is, $\Pi_U = \mathcal{P}(\mathring{p}, U,\nabla U)$ for some smooth map $\mathcal{P}$ with values in endomorphisms of $\mathbb{R}^N$, which again only depends on $\mathring{p}$ by translation invariance of $\Gamma$. It follows that $\Pi_U(x_U) = \Pi_U(x+U)$ is of the form $\varphi_1$.
\end{proof}

\subsection{First order expansion of $\phi$}
\label{sec:firstorder}

Recall that $\mathcal{K}= \ker L$ is the space of ($W^{2,2}$) Jacobi fields, and let $\mathcal{K}^\perp$ be the orthocomplement of $\mathcal{K}$ in $L^2$. We consider a compactly supported normal vector field $U$ with $\|U\|_{C^2} <\epsilon_0$. Let $J= \pi_\mathcal{K}(U)$, $h=\pi_{\mathcal{K}^\perp}(U)$ be the respective projections so that $U=J+h$. Intuitively, first order expansion will show that $J$ is the dominant term. Note that in this subsection we do not assume anything about the rotation part $\pi_{\mathcal{K}_0}(U)$. 

The goal of this subsection is to use a first order expansion of $\phi$ to show that $U$ is dominated by $J$ and $\phi$ up to higher order (in $U$). This amounts to proving $L^2$ estimates for $h$ (and derivatives) of the correct order, which lead to $L^4$ estimates for $U$. (Note that $L^4$ will be the highest degree that $U$ appears when using the 1st order expansion.) 

%

We begin with the following first order Taylor expansion estimate:

\begin{lemma}
\label{lem:ord1exp}
There exists $C$ so that for any $U$ as above, we have the pointwise estimate
\begin{equation}
\label{eq:ord1exp}
|\phi_U - Lh| \leq C\langle x\rangle|U|_1^2 +C|\nabla^2 U|^2. 
\end{equation}
\end{lemma}
\begin{proof}
Let $\phi(s) = \mathcal{\varphi}(p,sU,s\nabla U, s \nabla^2 U)$. By the last assertion of Lemma \ref{lem:phi-smooth}, we have that $|\phi''(s)| \leq C\langle x\rangle|U|_1^2 + C|\nabla^2 U|^2$ (for any $s$). Taylor expansion about $s=0$ then gives 
\[|\phi(1) - \phi(0) - \phi'(0)| \leq C\langle x\rangle|U|_1^2 + C|\nabla^2 U|^2.\] Noting that $\phi(0)=0$, $\phi'(0) = \mathcal{D\varphi}(U) = LU= Lh$ completes the proof. 

%
\end{proof}

\begin{remark}
Using the naive bound $\|\varphi\|_{C^3} \leq C\langle x\rangle$ would instead give $|\phi_U - Lh| \leq C\langle x\rangle|U|_2^2$ (similar to \cite[Proposition 6.1]{CM19}). In particular, this coefficient of $|\nabla^2 U|$ grows in $\langle x\rangle$, which would cause issues. Indeed, the proof of Proposition \ref{prop:hW22} below is a bootstrapping argument, in which the Gaussian Poincar\'{e} inequality is used to move factors of $\langle x\rangle$ off $U$ terms, at the cost of involving higher derivatives. It is thus crucial to our argument that the coefficient of the highest derivative $|\nabla^2 U|$ in (\ref{eq:ord1exp}) is uniformly bounded. 

By a more careful analysis, as in \cite[Section 4, Appendix A]{CM15} (see also \cite[Lemma 6.17]{CM19} and the remarks there), one may in fact show the stronger estimate \[|\phi_U - Lh| \leq C\langle x\rangle|U|_1^2 + C |U|_1 |\nabla^2 U| \leq C\langle x\rangle |U|_1^2 + \frac{C}{\langle x\rangle} |\nabla^2 U|^2.\] This (and the improved estimates on the Hessian in the linear directions in Corollaries \ref{cor:rotation-est} and \ref{cor:jacobi-est}) seems to be essential for the arguments of \cite[Section 4]{CM15} and \cite[Section 6]{CM19}. 

In Proposition \ref{prop:hW22} below, by quantifying the `higher order' property of $h$ consistently by powers of $\|U\|_{L^2}$, we are able to complete the bootstrapping argument using the weaker estimate (\ref{eq:ord1exp}) (and without the sharpened estimates for the Hessian in linear directions of Corollaries \ref{cor:rotation-est} and \ref{cor:jacobi-est}). 
\end{remark}

We can now show that $h$ is of higher order: 

\begin{proposition}
\label{prop:hW22}
Let $\Gamma = \mathring{\Gamma}\times \mathbb{R}^{n-k}$. There exists $\epsilon_0$ and $C=C(m)$ so that if $U$ is a compactly supported normal field on $\Gamma$ with $\|U\|_{C^2}<\epsilon_0$, then for $m\in \mathbb{N}$ we have

 \begin{equation}
\label{eq:W22h}
\int_\Gamma \langle x\rangle^{2m} |h|_2^2 \rho \leq C(\|\phi_U\|_{W^{m,2}}^2 + \|U\|_{L^2}^4) ,\end{equation}
 \begin{equation}
\label{eq:W22U}
\int_\Gamma \langle x\rangle^{2m} |U|_2^2 \rho \leq C(\|\phi_U\|_{W^{m,2}}^2 + \|U\|_{L^2}^2) ,\end{equation}
\begin{equation}
\label{eq:W24U}
\int_\Gamma \langle x\rangle^{2m} |U|_2^4 \rho \leq C(\|\phi_U\|_{W^{m,2}}^2 + \|U\|_{L^2}^4) .\end{equation}

Moreover, for any $m\in\mathbb{N}$, $\kappa\in(0,1]$ and $q\in [1,\frac{3+\kappa}{3})$, for some $C(m,\kappa,q)$ we have
\begin{equation}
\label{eq:W23qU}
 \int_\Gamma \langle x\rangle^{2m} |U|_2^{3q} \rho \leq C(m,\kappa,q) ( \|\phi_U\|_{L^2}^\frac{6q}{3+\kappa} + \|U\|_{L^2}^{3q}).\end{equation}
\end{proposition}

\begin{proof}
We follow the proof structure of \cite[Proposition 6.29]{CM19}. 
We first prove the estimate (\ref{eq:W22h}) for $h$ in the case $m=0$:
Applying Lemma \ref{lem:elliptic-est}(b) to $h$ we have $\|h\|_{W^{2,2}}^2 \leq C \|Lh\|_{L^2}^2$. Now integrating Lemma \ref{lem:ord1exp}, we have \begin{equation}
\label{eq:Lh-est}
\begin{split}
\|Lh\|_{L^2}^2 &\leq \|\phi_U\|_{L^2}^2 + C\|\phi_U-Lh\|_{L^2}^2
\\&\leq \|\phi_U\|_{L^2}^2 + C\int_\Gamma \left(\langle x\rangle^2 |U|_1^4  + \langle x\rangle^{-2} |\nabla^2 U|^4\right) \rho \leq \|\phi_U\|_{L^2}^2 + C\int_\Gamma |U|_2^4 \rho.
\end{split}
\end{equation} For the last inequality we used the Gaussian Poincar\'{e} inequality Lemma \ref{lem:poincare}. 

Now by the triangle inequality we have
\[
|U|_2^4 \leq 2|U|_2^2 |h|_2^2 + 2|U|_2^2 |J|_2^2 \leq 2\epsilon_0^2 |h|_2^2 + 2|U|_2^2 |J|_2^2. 
\]
Using the absorbing inequality on the last term then gives
\begin{equation}
\label{eq:Utriangle}
|U|_2^4 \leq 4 \epsilon_0^2 |h|_2^2 + 4|J|_2^4.
\end{equation}

By Corollary \ref{cor:jacobi-est} we have \begin{equation}\label{eq:J2}|J|_2 \leq C\langle x\rangle^2 \|J\|_{L^2}\leq C\langle x\rangle^2 \|U\|_{L^2}.\end{equation} Integrating (\ref{eq:Utriangle}) gives \[\int_\Gamma |U|_2^4\rho \leq 4\epsilon_0^2 \|h\|_{W^{2,2}}^2 + C \|U\|_{L^2}^4.\] Combining this with the previous estimates for $h$ and absorbing the $\epsilon_0^2$ term, we have \begin{equation}
\label{eq:W22hm0} \|h\|_{W^{2,2}}^2 \leq C(\|\phi_U\|_{L^2}^2 + \|U\|_{L^2}^4).\end{equation}

Following the same strategy, we now prove (\ref{eq:W22h}) for the cases $m\geq 1$: 
Applying Lemma \ref{lem:elliptic-est}(a) to $h$ we have 
\[ \int_\Gamma\langle x\rangle^{2m} |h|_2^2 \rho \leq C\|h\|_{L^2}^2 + C\int_\Gamma \langle x\rangle^{2m} |Lh|^2\rho.\]

The first term on the right may be bounded by the $m=0$ case of (\ref{eq:W22h}). By Lemma \ref{lem:ord1exp}, we have \begin{equation}
\label{eq:ord1exp-conc}
|Lh|^2 \leq |\phi_U|^2 + C\langle x\rangle^2|U|_1^4 + C|\nabla^2 U|^4.
\end{equation}
Integrating this against $\langle x\rangle^{2m}\rho$, and using Lemma \ref{lem:poincare}, we then find that \[ \int_\Gamma \langle x\rangle^{2m} |Lh|^2\rho \leq \int_\Gamma \langle x\rangle^{2m} |\phi_U|^2 \rho + C\int_\Gamma \langle x\rangle^{2m} |U|_2^4 \rho.\] 

Now integrating (\ref{eq:Utriangle}) against $\langle x\rangle^{2m}\rho$ gives \begin{equation}\label{eq:U24} \int_\Gamma \langle x\rangle^{2m} |U|_2^4 \rho \leq 4\epsilon_0^2 \int_\Gamma \langle x\rangle^{2m} |h|_2^2 \rho + C(m) \|U\|_{L^2}^4.\end{equation} Absorbing the $\epsilon_0^2$ term as before and using Lemma \ref{lem:poincare} $m$ times gives 
\begin{equation}
\int_\Gamma \langle x\rangle^{2m} |h|_2^2 \rho \leq  C\|\phi_U\|_{L^2}^2 +  C \int_\Gamma \langle x\rangle^{2m} |\phi_U|^2 \rho + C\|U\|_{L^2}^4 \leq C( \|\phi_U\|_{W^{m,2}}^2 + \|U\|_{L^2}^4)
\end{equation}
as desired. The $W^{2,2}$ estimate (\ref{eq:W22U}) for $U$ follows, using (\ref{eq:J2}) to bound the $J$ contribution. The $W^{2,4}$ estimate (\ref{eq:W24U}) follows immediately from (\ref{eq:U24}). 

For the last estimate (\ref{eq:W23qU}), we argue similarly: By the triangle inequality \[
|U|_2^{3+\kappa} \leq 2|U|_2^{1+\kappa} |h|_2^2 + 2|U|_2^{1+\kappa} |J|_2^2 \leq 2\epsilon_0^{1+\kappa} |h|_2^2 + 2|U|_2^{1+\kappa} |J|_2^2. 
\]
Young's inequality gives for any $\delta>0$ that $|U|_2^{1+\kappa} |J|_2^2 \leq \delta |U|_2^{3+\kappa} + C(\delta,\kappa) |J|_2^{3+\kappa}$, so taking $\delta = \frac{1}{4}$ and absorbing into the left hand side gives, for some $C_\kappa$, 
\begin{equation}
|U|_2^4 \leq 4 \epsilon_0^2 |h|_2^2 + C_\kappa |J|_2^4.
\end{equation} Continuing the argument just as above, one may show that
\[\int_\Gamma  |U|_2^{3+\kappa} \rho \leq  C\|\phi_U\|_{L^2}^2 + C'_\kappa \|U\|_{L^2}^{3+\kappa}. \] Finally, as $3q<3+\kappa$, H\"{o}lder's inequality gives \[\int_\Gamma \langle x\rangle^{2m} |U|_2^{3q}\rho \leq \left( \int_\Gamma |U|^{3+\kappa}_2 \rho\right)^\frac{3q}{3+\kappa} \left(\int_\Gamma \langle x\rangle^{\frac{2m(3+\kappa)}{3+\kappa-3q}} \rho\right)^\frac{3+\kappa-3q}{3+\kappa},\] which completes the proof since the last integral is finite. 
\end{proof}

Before moving on to the second order expansion, we record a consequence for the expansion of $F$:

\begin{proposition}
\label{prop:F-exp}
Let $\Gamma = \mathring{\Gamma}\times \mathbb{R}^{n-k}$. There exists $C$ and $\epsilon_0$ so that if $U$ is a compactly supported normal field on $\Gamma$ with $\|U\|_{C^2}<\epsilon_0$, then 
\begin{equation}
|F(\Gamma_U) - F(\Gamma)| \leq C(\|\phi_U\|_{L^2} \|U\|_{L^2} + \|U\|_{L^2}^3).
\end{equation}
\end{proposition}
\begin{proof}
This follows essentially as in \cite[Proof of Proposition 6.5]{CM19}: Let $F(s):= F(\Gamma_{sU})$, so $F'(s) = -\int_{\Gamma_{sU}}\langle \phi_{sU}, U\rangle \rho$, and by the fundamental theorem of calculus, 
\[\begin{split}
F(\Gamma_U) - F(\Gamma) - \frac{1}{2}\langle U,LU\rangle_{L^2} &= \int_0^1 \frac{d}{ds}\left(F(t) -\frac{t^2}{2} \langle U, LU\rangle_{L^2}\right)\d s
\\&= \int_0^1 \langle U, \phi_{sU} - sLU\rangle_{L^2} \d s.
\end{split}\]

As the (weighted) area elements are uniformly equivalent up to $C|U|_1 < C\epsilon_0$, and $LU=Lh$, it follows that

\[\begin{split}|F(\Gamma_U) - F(\Gamma)| &\leq \frac{1}{2} \|U\|_{L^2} \|Lh\|_{L^2} + C\int_0^1  \|U\|_{L^2}\|\phi_{sU}-sLh\|_{L^2} \d s +  C \|Lh\|_{L^2} \||U| |U|_1\|_{L^2}
\\& \leq C\|U\|_{L^2} \|Lh\|_{L^2} + C\int_0^1  \|U\|_{L^2}\|\phi_{sU}-sLh\|_{L^2} \d s.
\end{split}\]

By (\ref{eq:Lh-est}) and (\ref{eq:W24U}) both terms are bounded by $C\|U\|_{L^2}(\|\phi_U\|_{L^2} + \|U\|_{L^2}^2)$.
\end{proof}

 \subsection{Second order expansion of $\phi$}

We continue with a compactly supported normal field $U$ satisfying $\|U\|_{C^2}<\epsilon_0$. Let $J,h$ be as before, and set $U_0 = \pi_{\mathcal{K}_0}(U)$ and $J' = J-U_0 \in\mathcal{K}_1$. 
 
In the remainder of this section we also assume that $B_1(\mathring{\Gamma})\neq 0$. Then by the second variation analysis, in particular Proposition \ref{prop:d2phiEst}, we have 
\begin{equation}
\label{eq:d2phiJ'}
\|\pi_\mathcal{K}(\mathcal{D}^2\mathcal{\varphi}(J',J'))\|_{L^2} \geq \delta \|J'\|_{L^2}^2.\end{equation}

The goal of this subsection is to use a second order expansion and (\ref{eq:d2phiJ'}) to show that $U$ is controlled, modulo its rotation component, by $\phi$. To do so, we now use second order Taylor expansion to derive an estimate for $\mathcal{D\varphi}(J',J')$: 
 
 \begin{lemma}
 \label{lem:ord2exp}
There exists $C$ so that for any $U$ as above, we have the pointwise estimate
 \begin{equation}
 |\phi_U - Lh - \frac{1}{2}\mathcal{D}^2\mathcal{\varphi}(J',J')| \leq C \langle x\rangle( |U|_2^3+ 2|J|_2|h|_2 +|h|_2^2 + 2|J'|_2|U_0|_2  + |U_0|_2^2).
 \end{equation}
\end{lemma}
\begin{proof}
Let $\phi(s) = \mathcal{\varphi}(p,sU,s\nabla U, s \nabla^2 U)$ as before. Using the naive estimate $\|\mathcal{\varphi}\|_{C^3} \leq C\langle x\rangle$, we have $|\phi^{(3)}(s)| \leq C\langle x\rangle |U|_2^3$ (for any $s$). Taylor expansion about $s=0$ then gives 
\[|\phi(1) - \phi(0) - \phi'(0) - \frac{1}{2}\phi''(0)| \leq C\langle x\rangle |U|_2^3.\] Note that $\phi(0)=0$, $\phi'(0) = \mathcal{D\varphi}(U) = LU= Lh$ and $\phi''(0)=\mathcal{D}^2\mathcal{\varphi}(U,U)$. Expanding $\mathcal{D}^2\mathcal{\varphi}(U,U)$ using bilinearity and using $\|\mathcal{\varphi}\|_{C^3} \leq C\langle x\rangle$ to estimate all terms except $\mathcal{D}^2\mathcal{\varphi}(J',J')$ completes the proof. 
\end{proof}

Again, if one naively takes the $L^2$ norm of both sides, it may be difficult to control higher order terms like $\|\langle x \rangle |U|_2^3\|_{L^2} = \left(\int_\Gamma \langle x\rangle^2 |U|_2^6\right)^\frac{1}{2}$. We overcome this by first projecting to the finite dimensional space $K$ and using the explicit control on Jacobi fields to replace the $L^2$ norm with an $L^q$ norm, $q\in (1,2)$. This will have the additional benefit of reducing the degree to which $\phi$ enters the estimate, which will lead to a smaller cutoff error in later sections. 
 
Indeed, by symmetry of $L$ (or just integrating by parts), we have that $Lh\in \mathcal{K}^\perp$, so 
\begin{equation}
\label{eq:piKd2phi1}
\frac{1}{2} \| \pi_\mathcal{K}(\mathcal{D}^2\mathcal{\varphi}(J',J'))\|_{L^2} \leq \|\pi_\mathcal{K}(\phi_U)\|_{L^2} + \|\pi_\mathcal{K}(\phi_U - Lh - \frac{1}{2}\mathcal{D}^2\mathcal{\varphi}(J',J'))\|_{L^2}.
\end{equation}

We now estimate the size of the projection by the $L^q$ norm:

\begin{lemma}
\label{lem:L2Lq}
For any $q>1$ there exists $C(q)$ such that $\|\pi_\mathcal{K}(W)\|_{L^2} \leq C(q) \|W\|_{L^q}$. 
\end{lemma}
\begin{proof}

Let $V_i$ be an orthonormal basis of the finite dimensional space $\mathcal{K}$. Then there exists $C$ such that $|V_i|\leq C\langle x\rangle^2$ for any $i$. By H\"{o}lder's inequality, for any $q>1$, we have 
\[ \left| \int_\Gamma \langle W, V_i\rangle \rho \right| \leq C \int_\Gamma |W| \langle x\rangle^2\rho  \leq C(q) \|W\|_{L^q}.\]
\end{proof}

We are now in a position to prove the main estimate of this section, which shows that nearby graphs $\Gamma_U$ are either dominated by their rotational component or controlled by their closeness to a shrinker:

\begin{theorem}
\label{thm:entire-est}
Let $\Gamma = \mathring{\Gamma}\times \mathbb{R}^{n-k}$, where $\mathring{\Gamma}$ is a closed shrinker satisfying (A1-A2) and $B_1(\mathring{\Gamma})\neq 0$. Given $\kappa\in(0,1]$, $q \in(1,\frac{3+\kappa}{3})$, there exists $\epsilon_0>0$ and $C=C(\kappa,q)$ such that if $U$ is a compactly supported normal vector field on $\Gamma$ with $\|U\|_{C^2}\leq \epsilon_0$ and $U_0 = \pi_{\mathcal{K}_0}(U)$, then 
  \begin{equation}
 \|U\|_{L^2}^2 \leq C (\|U_0\|_{L^2}^2 + \|\phi_U\|_{L^q} + \|\phi_U\|_{W^{1,2}}^\frac{2}{q} + \|\phi_U\|_{L^2}^\frac{6}{3+\kappa}).
 \end{equation} 
\end{theorem}

\begin{proof}

By Lemma \ref{lem:L2Lq}, (\ref{eq:piKd2phi1}) implies that
\begin{equation}
\label{eq:piKd2phi2}
\frac{1}{2} \| \pi_\mathcal{K}(\mathcal{D}^2\mathcal{\varphi}(J',J'))\|_{L^2} \leq C\left(\|\phi_U\|_{L^q} + \|\phi_U - Lh - \frac{1}{2}\mathcal{D}^2\mathcal{\varphi}(J',J')\|_{L^q}\right). 
\end{equation}

We use Lemma \ref{lem:ord2exp} to estimate the last term:

\begin{equation} \|\phi_U - Lh - \frac{1}{2}\mathcal{D}^2\mathcal{\varphi}(J',J')\|_{L^q}  \leq  C\left(\int_\Gamma  \langle x\rangle^q \left(|U|_2^{3q} + |J|_2^q |h|_2^q + |h|_2^{2q} + |J'|_2^q | U_0|_2^q + |U_0|_2^{2q} \right)\rho\right)^\frac{1}{q} .
\end{equation}

For terms involving Jacobi fields, we use Corollaries \ref{cor:jacobi-est} and \ref{cor:rotation-est} respectively to bound $|J|_2, |J'|_2 \leq C\langle x\rangle^2 \|U\|_{L^2}$ and $|U_0|_2 \leq C \langle x\rangle \|U_0\|_{L^2}$.

For the first term, estimate (\ref{eq:W23qU}) from the first order expansion gives $\int_\Gamma \langle x\rangle^q |U|_2^{3q} \leq C(\|\phi_U\|_{L^2}^\frac{6q}{3+\kappa} + \|U\|_{L^2}^{3q}).$ For the second term, H\"{o}lder's inequality and estimate (\ref{eq:W22h}) gives \[\|\langle x\rangle |J|_2|h|_2\|_{L^q} \leq C\|U\|_{L^2} \|\langle x\rangle^3 |h|_2 \|_{L^q} \leq C(q)\|U\|_{L^2} \|h\|_{W^{2,2}} \leq C\|U\|_{L^2}(\|\phi_U\|_{L^2} + \|U\|_{L^2}^2) .\] 

For the third term, we first use Minkowski's inequality, so that for small enough $\epsilon_0$, \[|h|_2^{2q} \leq |h|_2^2( |J|_2^{2q-2} + |U|_2^{2q-2}) \leq (\|U\|_{L^2}^{2q-2} +\epsilon_0^{2q-2})|h|_2^2 \leq |h|_2^2,\] then estimate (\ref{eq:W22h}) again gives $\int_\Gamma \langle x\rangle^q |h|_2^{2q}\leq \int_\Gamma \langle x\rangle^2 |h|_2^2 \leq C(\|\phi_U\|_{W^{1,2}}^2+ \|U\|_{L^2}^4)$. 

Together, these estimates give

 \begin{equation}
 \begin{split}
\frac{1}{C}\|J'\|_{L^2}^2\leq{} &   \|\phi_U\|_{L^q} + \|\phi_U\|_{L^2}^\frac{6}{3+\kappa} + \|U\|_{L^2}^3  + \|U\|_{L^2}( \|\phi_U\|_{L^{2}} + \|U\|_{L^2}^2) \\& + \|\phi_U\|_{W^{1,2}}^\frac{2}{q} + \|U\|_{L^2}^\frac{4}{q}  + \|U\|_{L^2}\|U_0\|_{L^2} + \|U_0\|_{L^2}^2.
\end{split}
\end{equation}

Now, for $\epsilon\in(0,1)$ to be chosen later, we estimate $\|U\|_{L^2} \|\phi_U\|_{L^2} \leq \epsilon \|U\|_{L^2}^2 + \frac{1}{4\epsilon} \|\phi_U\|_{L^2}^2$ and $\|U\|_{L^2} \|U_0\|_{L^2} \leq \epsilon \|U\|_{L^2}^2 + \frac{1}{4\epsilon} \|U_0\|_{L^2}^2$. For small enough $\epsilon_0$ we certainly have $ \|\phi_U\|_{L^q}, \|\phi_U\|_{W^{1,2}}<1$, so lower powers of $\phi_U$ dominate (note that $3q<3+\kappa\leq 4$). For powers of $U$, we keep more careful track, using $\|U\|_{L^2} \leq \sqrt{\lambda(\Gamma)} \epsilon_0$. This gives a simplified estimate 
 
 \begin{equation}
 \label{eq:J'-final}
 \|J'\|_{L^2}^2 \leq C(\kappa,q) (C_\epsilon\|U_0\|_{L^2}^2 +(\epsilon+\epsilon_0) \|U\|_{L^2}^2 + \|\phi_U\|_{L^q} + \|\phi_U\|_{W^{1,2}}^\frac{2}{q} +C_\epsilon \|\phi_U\|_{L^2}^\frac{6}{3+\kappa}). 
 \end{equation}
 
Finally, we will use (\ref{eq:J'-final}) to estimate $J'$ in the expansion \begin{equation}\label{eq:U-final}\|U\|_{L^2}^2 = \|U_0\|^2_{L^2} + \|h\|_{L^2}^2 + \|J'\|_{L^2}^2 \leq \|U_0\|_{L^2}^2  + C(\|\phi_U\|_{L^2}^2  +\|U\|_{L^2}^4) + \|J'\|_{L^2}^2. \end{equation}
Indeed, if we choose $\epsilon+\epsilon_0< \frac{1}{C(\kappa,q)}$
 then the $\|U\|_{L^2}^2$ term in (\ref{eq:J'-final}) may be absorbed into the left hand side of (\ref{eq:U-final}). This gives the final estimate

\begin{equation}
 \|U\|_{L^2}^2 \leq C'(\kappa,q) (\|U_0\|_{L^2}^2 + \|\phi_U\|_{L^q} + \|\phi_U\|_{W^{1,2}}^\frac{2}{q}+ \|\phi_U\|_{L^2}^\frac{6}{3+\kappa}).
\end{equation} 
\end{proof}

 \section{Rotation}
\label{sec:rotation}

In this section we analyse changes in graph representation upon rotation of the base submanifold. The notation follows that of Section \ref{sec:rot0}. 

The goal is to show that given two sufficiently close submanifolds $\Sigma, \Gamma$, one may write $\Sigma$ as a normal graph $V$ over a rotation of $\Gamma$ so that $U$ is orthogonal to the space of rotation vector fields $\mathcal{K}_0$ on $\Gamma$. We also need to perform this rotation in a quantitative way, estimating the size of $V$ relative to the original distance between $\Sigma,\Gamma$. When $\Gamma$ is compact, this is an easy consequence of the slice theorem for group actions (see \cite{ELS18}, also \cite{SZ20}). 

For noncompact $\Gamma$, we only assume closeness on a large ball $B_R$, and it is also desirable to obtain estimates that depend explicitly on $R$. This creates a number of technical issues so we describe the change of graph operation by a somewhat arduous process, although we will be relatively cavalier about the domains of the graphs. We first write $\Sigma$ as a graph $U$ over $\Gamma$ and choose the rotation of $\Gamma$ to be exactly the one generated by $\pi_{\mathcal{K}_0}(U)$. The rotated graph $V$ will only be \textit{almost} orthogonal to rotations, but since we choose the appropriate rotation at the linear level, we are able to show that $\pi_{\mathcal{K}_0}(V)$ is higher order. These results are summarised in Proposition \ref{prop:graph-rot}. 

\subsection{Graphs over change of base submanifold}

Let $\Sigma^n$ be a smooth compact embedded submanifold of $\mathbb{R}^N$, denote by $C^k(N\Sigma)$ the space of $C^k$ normal vector fields on $\Sigma$, and let $C^k_\delta(N\Sigma)$ be the subset with $C^k$ norm at most $\delta$. Fix an orthonormal frame $\{e_i\}$ on $T\Sigma$. 

Consider the map $\Phi: \Sigma \times \Sigma\times C^2(N\Sigma) \times C^3(N\Sigma)\times \mathbb{R}^N \to \mathbb{R}^N \times \mathbb{R}^n$ defined by \begin{equation}\label{eq:impPhi}\Phi(p,q,U,V,W) = (X_V(p) +W -  X_U(q) , (\langle W, (X_V)_*e_i(p))_{i=1,\cdots,n}).\end{equation} Note $\Phi$ is defined so that $\Phi=0$ means that the graphs of $W$ over $\Sigma_V$ (at $X_V(p)$) and of $U$ over $\Sigma$ (at $q$) coincide, and $W$ is normal to $\Sigma_V$ at $X_V(p)$. 

The map $\Phi$ is $C^2$ and for any $p_0$ in the interior of $\Sigma$ we calculate \[(D\Phi)_{(p_0,p_0,0,0,0)}(Y,Z,U',V',W') =( Y +  V'(p_0) + W' -Z - U'(p_0), (\langle W', e_i(p_0)\rangle)_{i=1,\cdots,n}  ).\] Here $Y,Z \in T_{p_0}\Sigma$. In particular $(D\Phi)_{(p_0,p_0,0,0,0)}(0,Z ,0,0,W')$ is invertible, and by the implicit function theorem there are (unique) $C^2$ maps $\mathcal{Q}=\mathcal{Q}^{\Sigma,p_0},\mathcal{W}=\mathcal{W}^{\Sigma,p_0}$ defined on a neighbourhood of $(p_0,0,0)$ in $\Sigma\times C^2(N\Sigma)\times C^3(N\Sigma)$ so that \begin{equation}\label{eq:local-QW}\Phi(p, \mathcal{Q}(p,U,V) , U,V, \mathcal{W}(p,U,V))=0.\end{equation}

Globally, on a submanifold $\Gamma$ we may patch together the local constructions above on each intrinsic ball $\Sigma = B^\Gamma_r(p)$ using uniqueness (if $\Gamma$ is immersed, we simply take $r$ small enough so that $\Sigma$ is an embedded disk). The resulting constants will be uniform if $\Gamma$ is closed, or a cylinder over a closed submanifold. We summarise this as follows:

\begin{proposition}
\label{prop:global-graph}
Fix a closed submanifold $\mathring{\Gamma}$ and set $\Gamma =\mathring{\Gamma}\times \mathbb{R}^{n-k}$. There exist $r\in(0,\frac{1}{2})$, $K$ and a map $\mathcal{W}(p,U,V)$, such that for any $\epsilon>0$, there is a $\delta>0$ so that for any $p\in\Gamma\cap B_{R-1/2}$ and normal vector fields $U,V$ on $\Gamma\cap B_R$ with $\|U\|_{C^2}, \|V\|_{C^3}<\delta$, the following hold:
\begin{itemize}
\item There exists $p_0$ so that $p\in B^\Gamma_{r/2}(p_0)$ and $\mathcal{W}(p,U,V) = \mathcal{W}^{\Sigma,p_0}(p,U|_\Sigma, V|_\Sigma)$, where $\Sigma = B^\Gamma_r(p_0)$ and the map $\mathcal{W}^{\Sigma,p_0}: B^\Gamma_{r/2}(p_0) \times C^2_\delta(N\Sigma) \times C^3_\delta(N\Sigma)\to \mathbb{R}^N$ satisfies $\|\mathcal{W}^{\Sigma,p_0}\|_{C^2} \leq K$; 
\item $D\mathcal{W}_{(p,0,0)}(Y, U',V') = - V'(p) + U'(p)$;
\item $W(p)=\mathcal{W}(p,U,V)$ is normal to $\Gamma_V$ at $X_V(p)$;
\item For any $s\leq R$ we have $\Gamma_U\cap B_{s-2} \subset (\Gamma_V)_W \cap B_{s-1} \subset \Gamma_U\cap B_s$. 
\item As a map on $\Gamma\cap B_{R-1/2}$ we have $\|W\|_{C^2}<\epsilon$.

\end{itemize}
\end{proposition}
\begin{proof}
All but the last point follow immediately from the local construction. The last point follows by continuity, noting that $q=\mathcal{Q}(\cdot,U,V)$ and $W=\mathcal{W}(\cdot,U,V)$ are perturbations of $\mathcal{Q}(\cdot,0,0)=\id$ and $\mathcal{W}(\cdot,0,0)\equiv 0$. 
\end{proof}

The key is that the above construction does not depend on $R$. 

\subsection{Rotated cylinders}

Here we write a rotated cylinder as a graph over the original.

\begin{proposition}
\label{prop:rot-cyl}
Let $\Gamma = \mathring{\Gamma}\times \mathbb{R}^{n-k}$. There is a neighbourhood $\mathcal{U}$ of the identity in $\mathrm{SO}(N)$ and a smooth map $\mathcal{V}: \mathcal{U} \times \Gamma \to \mathbb{R}^N$ such that:
\begin{itemize}
\item For each $\Theta\in \mathcal{U}$, the vector field $V_\Theta:= \mathcal{V}(\Theta,\cdot)$ is normal on $\Gamma$ and its graph coincides with the rotation $\Theta\cdot\Gamma$;
\item For any $m$ and at any $\Theta,p$, we have $\|\mathcal{V}\|_{C^m} \leq c_m d(\Theta,\id) \langle X(p)\rangle$. 
\end{itemize}
\end{proposition}

\begin{proof}

We may certainly construct the desired map $\mathcal{V}$ on any bounded portion $\mathring{\Gamma}\times B^{n-k}_{r_0}$ by applying the implicit function theorem locally as above. For instance one may consider the smooth map $\Psi(p,V, q, \Theta)=( X(p)+V-\Theta\cdot X(q) , (\langle V, X_i(p))_{i=1,\cdots,n})$ on each $\Sigma = B^\Gamma_r(p)$, and solve $\Psi=0$ for $V,q$ as before. 

The construction extends, using translation invariance, to all of $\Gamma =\mathring{\Gamma}\times \mathbb{R}^{n-k}$ as follows. Recall that $\bar{\pi}$ is the projection $\mathbb{R}^N\to \mathbb{R}^{n-k}$ to the linear directions. Set $\bar{\pi}^\perp(x) = x-\bar{\pi}(x)$. 

For $\Theta$ in a neighbourhood $\mathcal{U}$ of the identity, $\bar{\pi}\circ\Theta$ restricts to a linear isomorphism $\mathbb{R}^{n-k}\to\mathbb{R}^{n-k}$. Consider $p=(\mathring{p},y) \in \Gamma = \mathring{\Gamma}\times \mathbb{R}^{n-k}$. Take $q=(\mathring{q}, z_0)$ so that (as constructed above) $X(\mathring{p},0)+V_\Theta(\mathring{p},0) = \Theta \cdot X(q)$. Let $z$ be such that $\bar{\pi}(\Theta(z))=y$; then \[\begin{split}\Theta\cdot X(\mathring{q}, z_0 + z) &= \Theta \cdot (X(q) + z) = X(\mathring{p},0)+V_\Theta(\mathring{p},0) + \Theta(z) \\&= X(p) + V_\Theta(\mathring{p},0) + \bar{\pi}^\perp(\Theta(z)).\end{split}\]
In particular, we may extend $V_\Theta$ by defining $V_\Theta(\mathring{p},y) = V_\Theta(\mathring{p},0) + \bar{\pi}^\perp(\Theta\cdot (\bar{\pi}\circ\Theta)^{-1}(y))$. 

This extension clearly depends smoothly on $\Theta$, and grows linearly in $y$. The implies the desired estimates for sufficiently small $\mathcal{U}$. 

%
%
\end{proof}

Recall that we also defined the normal fields $J_\theta (p)  := \Pi( \pr_s |_{s=0} \exp(s\theta)(p))$ for $\theta\in\mathfrak{so}(N)$ in Section \ref{sec:rot0}. Clearly, $J_\theta$ is the linearisation of $V_{\Theta}$, that is $\pr_s|_{s=0} V_{\exp (s\theta)} = J_\theta$. 

\subsubsection{Change of basepoint}

We also briefly account for the change in basepoint between the parametrisations of $\Theta\cdot \Gamma$ as a rotation of $\Gamma$, and as a graph over $\Gamma$. 

First consider a compact embedded submanifold $\Sigma$ and let $\pi_\Sigma$ be the nearest point projection to $\Sigma$. For any interior point $x$ of $\Sigma$, there is some ball $B_\delta(x)$ on which $\pi_\Sigma$ is smooth. For $\Theta\in \mathrm{SO}(N)$ in a neighbourhood of the identity, its action as a diffeomorphism on $\mathbb{R}^N$ satisfies $\|\Theta-\id\|_{C^m} \leq c_m d(\Theta,\id) \langle x\rangle$. In particular, if $d(\Theta, \id) \langle x\rangle$ is small enough, then $|\Theta(x)-x|<\delta$ and $p=\pi_\Sigma(\Theta(x))$ satisfies $X_{V_\Theta}(p) = \Theta(x)$. 

Again, on $\Gamma = \mathring{\Gamma}\times \mathbb{R}^{n-k}$ we may patch together this local construction on intrinsic balls $\Sigma = B^\Gamma_r(p_0)$, each of which is an embedded disk. The projections $\pi_\Sigma$ will then satisfy a uniform $C^m$ bound, so the change of basepoint has small $C^m$ norm so long as $d(\Theta,\id) \langle x\rangle$ is small. We summarise this as follows:
\begin{proposition}
\label{prop:basepoint}
Fix a closed submanifold $\mathring{\Gamma}$ and set $\Gamma = \mathring{\Gamma}\times \mathbb{R}^{n-k}$. 

Let $\mathcal{S}_\delta = \{(p,\Theta)\in \Gamma \times \mathrm{SO}(N) | \langle X(p)\rangle d(\Theta,\id) < \delta\}$. There exist $r\in(0,\frac{1}{2})$, $K$, $\delta_0>0$ and a smooth map $\mathcal{P}: \mathcal{S}_{\delta_0} \rightarrow \Gamma$ such that:

\begin{itemize}
\item For any $p\in \Gamma$ there exists $p_0$ so that $p\in B^\Gamma_{r/2}(p_0)$ and \[\mathcal{P}(p,\Theta) = X^{-1}(\pi_{B^\Gamma_r(p_0)}(\Theta(X(p))));\] 
\item For any $\epsilon>0$, there exists $\delta>0$ such that $\|\mathcal{P}\|_{C^2}<\epsilon$ on $\mathcal{S}_\delta$.
\end{itemize}
\end{proposition}

Note that we have composed with the immersion $X:\Gamma\to\mathbb{R}^N$ and its (local) inverse to ensure that $\mathcal{P}(\Theta,\cdot)$ is indeed a map $\Gamma\to\Gamma$.

\subsection{Graphs over rotated cylinders}

We now quantitatively rewrite a graph as a graph over a rotated cylinder. The key is that the constants do not depend on the size of the graphical domain. For $\Theta \in \mathrm{SO}(N)$ the pullback $\Theta^*$ is an isometry identifying the normal bundles $N(\Theta \cdot \Gamma)  \simeq N\Gamma$, with inverse $(\Theta^{-1})^* = \Theta_*$. It will again be convenient to perform the construction on the original cylinder $\Gamma$, noting that $\Theta\cdot \Gamma_W = (\Theta\cdot\Gamma)_{\Theta_* W}$. 

\begin{proposition}
\label{prop:graph-rot}
Let $\Gamma =\mathring{\Gamma}\times \mathbb{R}^{n-k}$. Given $\epsilon>0, C_l$, there exist $C,\delta,R_0>0$ such that the following holds: For any $R\geq R_0$, if $U$ is a normal vector field on $\Gamma \cap B_R$, with $\|U\|_{C^{2}} \leq \delta$ and $\|U\|_{L^2} \leq \frac{\delta}{R}$, then there is a rotation $\Theta \in \mathrm{SO}(N)$ and a normal vector field $W$, defined over $\Gamma \cap B_{R-2}$, such that
\begin{itemize}
\item for any $r\leq R$, the graphs satisfy \[ \Gamma_U \cap B_{r-3} \subset \Theta\cdot \Gamma_W \cap B_{r-2} \subset \Gamma_U \cap B_{r-1};\]
\item $\|W\|_{C^2}<\epsilon$; 
\end{itemize}
Furthermore, if $\|U\|_{C^{l+2}} \leq C_l$, $l>0$, then
\begin{itemize}
\item $\|W\|_{L^2}^2 \leq C(\|U\|_{L^2}^2 + \|U\|_{L^4}^{4a_l})$;
\item For any $J_0\in \mathcal{K}_0$ with $\|J_0\|_{L^2}=1$, we have \[\left|\int_{\Gamma\cap B_{R-3}} \langle W,J_0\rangle \rho\right| \leq C( \|U\|_{L^2}^{2a_l}+ \|U\|_{L^2} (R-3)^{n/2} \e^{-(R-3)^2/8}),\]
\end{itemize}
where $a_l=a_{l+1,2,n}= \frac{l}{l+2+n}$.
\end{proposition}
\begin{proof}
The outline of the construction is as follows: Consider $p\in\Gamma$; the point $\Theta\cdot X(p)$ on the rotated cylinder $\Theta\cdot\Gamma$ lies over the basepoint $\mathcal{P}(\Theta,p)$ on the original cylinder $\Gamma$. Using Proposition \ref{prop:global-graph}, we may write $\Gamma_U$ as a normal graph over the rotated cylinder (parametrised as $\Gamma_{V_\Theta}$). Switching to the rotated parametrisation and rotating the vector field back will give the desired normal graph over the original cylinder.

Concretely, set \[\wt{\mathcal{W}}(p,\Theta, U) = \Theta^{-1} \cdot \mathcal{W}( \mathcal{P}(\Theta,p), U,V_\Theta),\] where $\mathcal{W}$, $V_\Theta = \mathcal{V}(\Theta,\cdot)$ and $\mathcal{P}$ are as constructed in Propositions \ref{prop:global-graph}, \ref{prop:rot-cyl} and \ref{prop:basepoint} respectively. If $\delta>0$ is sufficiently small, the map $\wt{\mathcal{W}}$ is well-defined for $p\in \Gamma \cap B_{R-2}$ and $U,\Theta$ such that $\|U\|_{C^2}<\delta$, $d(\Theta, \id) < \delta/R$. Moreover, we will have $(D\wt{\mathcal{W}})_{(p,\id, 0)} (0,\theta, U' ) = U'(p) - J_\theta(p)$ and the uniform estimates $\|\mathcal{W}\|_{C^2} \leq K$, $\|\mathcal{W}(\cdot,\Theta,U)\|_{C^2(B_{R-2})} \leq \epsilon$. 

Extend $U$ by $0$ to all of $\Gamma$, and set $\theta= \iota(\pi_{\mathcal{K}_0}(U))$, so that $J_\theta = \pi_{\mathcal{K}_0}(U)$. (Recall the notions in Section \ref{sec:rot0}; in particular recall that the metrics on $\mathfrak{so}(N)$ and $\mathrm{SO}(N)$ are chosen to that $\exp$ coincides with the Riemannian exponential map.) Define the one parameter family $W(s)(p) = \wt{\mathcal{W}}(p, \Theta(s), sU)$, where $\Theta(s):=\exp(s\theta)$; the desired normal vector field is then $W=W(1)$, and it satisfies the first two conclusions. Note that since $\exp$ is a radial isometry, we indeed have \[\label{eq:Theta-est}d(\Theta(s),\id) = |s| \|J_\theta\|_{L^2} \leq |s| \|U\|_{L^2} \leq |s| \delta /R.\]

It remains to prove the last two estimates. Recall that by the local construction of $\mathcal{W}$ (Proposition \ref{prop:global-graph}), for fixed $p$ the quantity $\wt{\mathcal{W}}(p,\Theta,U)$ depends only on the restriction $U|_{B_r(p)}$. In particular we calculate $W'(0) = U- J_\theta = \pi_{\mathcal{K}_0^\perp}(U)$ and $|W''(s)| \leq K(\langle x\rangle\|\theta\| + \|U\|_{C^2(B_r(p))})$, since $V_\Theta$ grows linearly. By choice of metrics, $\|\theta\| = \|\pi_{\mathcal{K}_0}(U)\|_{L^2}\leq \|U\|_{L^2}$. Thus by Taylor expansion about $s=0$, we have
\begin{equation}
\label{eq:Wexp}
|W-\pi_{\mathcal{K}_0^\perp}(U)| \leq K( \langle x\rangle^2 \|U\|_{L^2}^2 + \|U\|_{C^2(B_r(p))}^2). 
\end{equation}

For the last term we interpolate using Appendix \ref{sec:interpolation}. In particular, $a_l=a_{l,2,n}$, and by (\ref{eq:interp-est2}), \begin{equation}\label{eq:Wexp-err} \int_{\Gamma \cap B_{R-1}} \|U\|_{C^2(B_r(p))}^4 \rho(x(p)) \d p \leq C(r,l) \|U\|_{L^4}^{4a_l}.\end{equation}

By integrating (\ref{eq:Wexp}) and using (\ref{eq:Wexp-err}), it follows that $\|W\|_{L^2}^2 \leq C(\|U\|_{L^2}^2 + \|U\|_{L^4(B_R)}^{4a_l}).$ Now suppose $J_0\in\mathcal{K}_0$ is such that $\|J_0\|_{L^2}=1$. By (\ref{eq:Wexp}) and Corollary \ref{cor:rotation-est} we have \begin{equation}\label{eq:Wexp2}|\langle W,J_0\rangle - \langle \pi_{\mathcal{K}_0^\perp}(U),J_0\rangle | \leq C\langle x\rangle( \langle x\rangle^2 \|U\|_{L^2}^2 + \|U\|_{C^2(B_r(p))}^2).\end{equation} Note that the constant above does not depend on $J_0$. By (\ref{eq:interp-est2}) again, we have \[\int_{\Gamma \cap B_{R-1}} \langle x(p)\rangle\|U\|_{C^2(B_r(p))}^2 \rho(x(p)) \d p  \leq C\|U\|_{L^2}^{2a_l}.\] 
Then since $\int_\Gamma \langle \pi_{\mathcal{K}_0^\perp}(U),J_0\rangle \rho =0$, it follows from integrating (\ref{eq:Wexp2}) that  \[ \left|\int_{\Gamma \cap B_{R-3}} \langle W,J_0\rangle \rho\right| \leq \left| \int_{\Gamma \setminus B_{R-3}} \langle \pi_{\mathcal{K}_0^\perp}(U), J_0\rangle \rho\right| + C\|U\|_{L^2}^{2a_l} . \]
By Corollary \ref{cor:rotation-est} the first term on the right is bounded by $C\|U\|_{L^2} \left(\int_{\Gamma \setminus B_{R-3}} \langle x\rangle^2 \rho\right)^\frac{1}{2}$. By Lemma \ref{lem:cutoff} this is bounded by $C\|U\|_{L^2} (R-3)^{n/2} \e^{-(R-3)^2/8}$, which completes the proof. 
\end{proof}

\begin{remark}
One may prove various sharper estimates for (\ref{eq:Wexp-err}), by using weaker $C^0$ topology for $U$ in the construction of $\mathcal{W}$, or a more precise interpolation in Appendix \ref{sec:interpolation}. We have written this somewhat more naive estimate to streamline the exposition; using the $C^2$ topology guaranteed that $\mathcal{W}$ was $C^2$ in all variables, and we interpolate using higher derivatives for the main results anyway. 
\end{remark}

\section{{\L}ojasiewicz inequalities}
\label{sec:lojasiewicz}

In this section we prove the {\L}ojasiewicz inequalities using the estimates of Section \ref{sec:entire}, together with Section \ref{sec:rotation} to control the rotation part. We also prove the `improvement step' to be used later in applications. 

Since the estimates in Section \ref{sec:entire} apply for compactly supported graphs, we need to introduce a cutoff. This introduces cutoff errors which are exponential by Lemma \ref{lem:cutoff}. However, the cutoff error on an $L^q$ norm is larger for higher $q$ due to the exponent $\frac{1}{q}$. It will be crucial for the `improvement step' that we obtained the second order estimate Theorem \ref{thm:entire-est} with $\|\phi\|_{L^q}$, $q$ arbitrarily close to 1, instead of $\|\phi\|_{L^2}$. For convenience, set $\delta_R = R^n \e^{-R^2/4}$. 

We first prove a somewhat more general distance {\L}ojasiewicz inequality:

\begin{theorem}
\label{thm:lojasiewicz}
 Let $\mathring{\Gamma}$ be a closed shrinker satisfying (A1-A2) and $B_1(\mathring{\Gamma})\neq 0$. 
 
Given $q\in(1,\frac{4}{3}), \beta\in(\frac{1}{q},1)$, there exist $\epsilon_2>0, l_0$ so that the following holds: For any $\epsilon_1$, $\lambda_0$, $C_j$ there is an $R_0$ such that if $l\geq l_0$, $\Sigma^n\subset \mathbb{R}^N$ has $\lambda(\Sigma)\leq \lambda_0$ and
\begin{enumerate}
\item for some $R>R_0$, we have that $B_R\cap \Sigma$ is the graph of a normal field $U$ over some cylinder in $\mathcal{C}_n(\mathring{\Gamma})$ with $\|U\|_{C^{2}(B_R)} \leq \epsilon_2$ and $\|U\|_{L^2(B_R)} \leq \epsilon_2/R$, 
\item $|\nabla^j A| \leq C_j$ on $B_R\cap \Sigma$ for all $j \leq l$;
\end{enumerate}
then there is a cylinder $\Gamma \in\mathcal{C}_n(\mathring{\Gamma})$ and a compactly supported normal vector field $V$ over $\Gamma$ with $\|V\|_{C^{2,\alpha}} \leq \epsilon_1$, such that $\Sigma \cap B_{R-6}$ is contained in the graph of $V$, and 

\[ \|V\|_{L^2}^2 \leq C (\|U\|_{L^2}^{4a_l} +\|\phi\|_{L^q(B_R)} + \|\phi\|_{L^2(B_R)}^{2\beta}+ \delta_{R-5}^\beta  ),\]
where $C=C(n,\beta,q,l,C_l,\lambda_0, \epsilon_1)$, and $a_l \nearrow 1$ as $l\to \infty$. 
\end{theorem}
\begin{proof}
Fix a smooth cutoff function $\eta:\mathbb{R}\to [0,1]$ such that $\eta(t) =1$ for $t\leq 0$ and $\eta(t)=0$ for $t\geq 1$. Define $\eta_R:\mathbb{R}^N\to\mathbb{R}$ by $\eta_R(x) = \eta(|x|-R)$. 

Let $\epsilon_0>0$ be sufficiently small so that the results of Section \ref{sec:entire} apply. Take $\kappa \in (0,1)$ so that $\beta = \frac{3}{3+\kappa} <\frac{1}{q}$. Finally let $a_l = a_{l+2,2,n}$, $b_l = a_{l,1,n}$ be the exponents from interpolation (see Appendix \ref{sec:interpolation}). 

Let $\bar{\Gamma}$ be the original cylinder over which $U$ is defined. For $\epsilon_2$ small enough, $U$ will be defined at least over $\bar{\Gamma}\cap B_{R-1}$, and $U':=\eta_{R-2} U$ will satisfy $\|U'\|_{C^2} \leq \epsilon_0$ as well as
\begin{equation}\label{eq:containment0} \Sigma \cap B_{R-3} \subset \bar{\Gamma}_{U'} \cap B_{R-2} \subset \bar{\Gamma}_U \cap B_{R-1} \subset \Sigma\cap B_{R}.\end{equation}

Let $\delta>0$ be given by Proposition \ref{prop:graph-rot} with $\epsilon=\epsilon_0$. So long as $\epsilon_2<\delta$, we may apply Proposition \ref{prop:graph-rot} to $U$ on $B_{R-2}$ to obtain a cylinder $\Gamma \in \mathcal{C}_n(\mathring{\Gamma})$ and a normal vector field $W$ on $\Gamma \cap B_{R-4}$, and define $V= \eta_{R-5}W$. Then, possibly taking $\epsilon_2$ smaller still, we will have $\|W\|_{C^2} \leq \epsilon_0$, $\|V\|_{C^2} \leq \epsilon_0$ and \begin{equation}\label{eq:containment} \Sigma\cap B_{R-6} \subset \Gamma_V \cap B_{R-5} \subset  \Gamma_W \cap B_{R-4} \subset \Sigma\cap B_{R-3}.\end{equation}

We may now apply the estimates of Section \ref{sec:entire} to $V$ (and $U'$). We freely use that for small enough $\epsilon_2$ we certainly have $\|U\|_{L^2} , \|\phi\|_{L^q}, \|\phi\|_{L^2}<1$, so lower powers will be dominant. Note also that by (\ref{eq:containment0}), (\ref{eq:containment}) and the curvature bounds on $\Sigma$, we have uniform $C^{l+2}$ bounds on $U,W$ and hence $U', V$ respectively. 

Applying Theorem \ref{thm:entire-est} to $V=\eta_{R-5} W$, and interpolating using (\ref{eq:interp-est}), we have
  \begin{equation}
    \label{eq:lojasiewicz-entire}
    \begin{split}
 \|V\|_{L^2}^2 &\leq C(q) (\|\pi_{\mathcal{K}_0}(V)\|_{L^2}^2 + \|\phi_V\|_{L^q} + \|\phi_V\|_{W^{1,2}}^\frac{2}{q} + \|\phi_V\|_{L^2}^{2\beta})\\
& \leq C(q) (\|\pi_{\mathcal{K}_0}(V)\|_{L^2}^2 + \|\phi_V\|_{L^q} + \|\phi_V\|_{L^2}^\frac{2b_l}{q} + \|\phi_V\|_{L^2}^{2\beta}).
 \end{split}
\end{equation}

It remains to estimate the rotation part of $V$. Applying Proposition \ref{prop:hW22} to $U'$ implies 
 \begin{equation}\|U\|_{L^{4}(B_{R-2})}^4\leq \|U\|_{W^{2,4}(B_{R-2})}^4 \leq C(\|\phi_{U'}\|_{L^2}^2 + \|U\|_{L^2}^4).\end{equation} 

Then by the properties of $W$ from Proposition \ref{prop:graph-rot} (recall that we applied it on $B_{R-2}$), \begin{equation}\begin{split} \|V\|_{L^2}^2 &\leq \|W\|_{L^2}^2 \leq C( \|U\|_{L^2}^2 + \|U\|_{L^4(B_{R-2})}^{4a_l}) \\& \leq C'( \|U\|_{L^2}^{2} + \|U\|_{L^2}^{4a_l} + \|\phi_{U'}\|_{L^2}^{2a_l} ),\end{split}\end{equation} 
and if $J_0\in \mathcal{K}_0$ with $\|J_0\|_{L^2}=1$, then \begin{equation}\left|\int_{\Gamma\cap B_{R-5}} \langle V,J_0\rangle \rho\right| \leq C(\|U\|_{L^2(B_{R-2})}^{2a_l} +\delta_{R-5}^\frac{1}{2} \|U\|_{L^2} ).\end{equation} 

But by Corollary \ref{cor:rotation-est}, using Cauchy-Schwarz and Lemma \ref{lem:cutoff} we have \begin{equation}\left| \int_{\Gamma \setminus B_{R-5}} \langle V,J_0\rangle\rho \right| \leq C \|V\|_{L^2} \left(\int_{\Gamma \setminus B_{R-5}} \langle x\rangle^2 \rho \right)^\frac{1}{2} \leq C'\delta_{R-5}^\frac{1}{2} \|V\|_{L^2} .\end{equation}

Since $\mathcal{K}_0$ is finite-dimensional it follows that
\begin{equation}
\label{eq:projV}
\|\pi_{\mathcal{K}_0}(V)\|_{L^2} \leq C(\|\phi_{U'}\|_{L^2}^{2a_l} + \|U\|_{L^2}^{2a_l} + \delta_{R-5}).
\end{equation}

Last, we estimate the shrinker quantities using $\phi=\phi_\Sigma$. The quantities $\phi_V, \phi_{U'}$ are at most linear in $x$ by the $C^2$ bounds on $V,U'$, and they coincide with $\phi$ on $B_{R-5}$ and $B_{R-2}$ respectively. Using Lemma \ref{lem:cutoff} to estimate the cutoff error, we have 
\begin{equation}
\label{eq:U'-cutoff}
\begin{split}
\|\phi_{U'}\|_{L^2}^2 &= \int_{\Gamma\cap B_{R-2}} |\phi_{U'}|^2 \rho +  \int_{\Gamma\cap B_{R-1} \setminus B_{R-2}} |\phi_{U'}|^2 \rho
\\&\leq \int_{\Sigma\cap B_{R}} |\phi|^2 \rho + C \int_{\Gamma\cap B_{R-1} \setminus B_{R-2}} \langle x\rangle^2 \rho
\\&\leq  \|\phi\|_{L^2(B_{R})}^2 + C' \lambda(\Gamma) \delta_{R-2} 
.
\end{split}\end{equation}
and similarly, for any $s\in [1,2]$,
 \begin{equation}\|\phi_{V}\|_{L^s}^s \leq \|\phi\|_{L^s(B_R)}^s +C(s)\delta_{R-5},\end{equation}

Finally, substituting these estimates into (\ref{eq:projV}) and then (\ref{eq:lojasiewicz-entire}) gives 
\begin{equation}
\label{eq:main-estimate}
\begin{split}
\|V\|_{L^2}^2  \leq {}& C( \|\phi\|_{L^2(B_R)}^{4a_l} + \delta_{R-2}^{2a_l} + \|U\|_{L^2}^{4a_l} + \delta_{R-5}^2) \\&+ C(  \|\phi\|_{L^q(B_R)} + \delta_{R-5}^\frac{1}{q} + \|\phi\|_{L^2(B_R)}^\frac{2b_l}{q}+\delta_{R-5}^\frac{b_l}{q}+  \|\phi\|_{L^2(B_R)}^{2\beta} + \delta_{R-5}^\beta).
\end{split}
\end{equation}

This estimate holds for any $l\geq 2$; we now take $l$ large just to simplify the expression. In particular, for large enough $l$ we have $2a_l >1$ and $b_l > \beta q$. Since $\frac{1}{q}>\beta$, keeping dominant terms we conclude that

\begin{equation}
 \|V\|_{L^2}^2 \leq C\left(  \|U\|_{L^2}^{4a_l} +\|\phi\|_{L^q(B_R)} + \|\phi\|_{L^2(B_R)}^{2\beta}+ \delta_{R-5}^\beta \right).
\end{equation}

The $C^{2,\alpha}$ bound for $V$ follows by interpolation as long as $R$ is large enough. 

\end{proof}

We now deduce the main {\L}ojasiewicz inequalities and the `improvement step' by making successive simplifications. First, assuming a tighter bound on $U$ yields Theorem \ref{thm:lojasiewicz-intro}:

\begin{proof}[Proof of Theorem \ref{thm:lojasiewicz-intro}]
Apply Theorem \ref{thm:lojasiewicz} and take $l$ large enough so that $a_l > \beta$. The $L^2$ hypothesis then implies that $\|U\|_{L^2}^{4a_l}$ is dominated by the exponential error term $\delta_{R-5}^\beta$. 
\end{proof}

 By choosing $q$ close to 1 we get the gradient {\L}ojasiewicz inequality Theorem \ref{thm:lojasiewicz-grad}:

\begin{proof}[Proof of Theorem \ref{thm:lojasiewicz-grad}]
Let $V$ be the compactly supported normal vector field defined on $\Gamma \in \mathcal{C}_n(\mathring{\Gamma})$ that was obtained in Theorem \ref{thm:lojasiewicz-intro}. By Proposition \ref{prop:F-exp} and Young's inequality, 
\begin{equation}
\begin{split}
|F(\Gamma_V)-F(\Gamma)|  &\leq C( \|\phi_V\|_{L^2}^\frac{3}{2} + \|V\|_{L^2}^3)
\\&\leq C(  \|\phi_V\|_{L^2}^\frac{3}{2}  + \|\phi\|_{L^q}^{\frac{3}{2}} + \|\phi\|_{L^2}^{3\beta}+ \delta_{R-5}^\frac{3\beta}{2}). 
\end{split}
\end{equation}
By H\"{o}lder's inequality we have $\|\phi\|_{L^q} \leq C(q) \|\phi\|_{L^2}$, and as in the proof of Theorem \ref{thm:lojasiewicz} we have $\|\phi_V\|_{L^2}^2 \leq \|\phi\|_{L^2}^2 + \delta_{R-5}$. Since $\Sigma\cap B_{R-6}$ is contained in the graph $\Gamma_V$, by Lemma \ref{lem:cutoff} we have $|F(\Sigma)-F(\Gamma_V)| \leq C\delta_{R-6}$. Collecting dominant terms, as $\beta>\frac{3}{4}$ we conclude
\begin{equation}
|F(\Sigma)-F(\Gamma)| \leq C(\|\phi\|_{L^2}^\frac{3}{2} + \delta_{R-5}^\frac{3}{4}). 
\end{equation}

\end{proof}

Finally we assume that $\|\phi\|_{L^2}$ is small relative to the scale on which we have graphicality:

\begin{theorem}[Improvement step]
\label{thm:improvement}
 Let $\mathring{\Gamma}$ be a closed shrinker satisfying (A1-A2) and $B_1(\mathring{\Gamma})\neq 0$. 
There exists $\epsilon_2>0, l$ such that given $\epsilon_1>0$, $\theta>0$, $\lambda_0$, $C$, $C_j$, there exists $R_0$ so that the following holds. Suppose that $\Sigma^n\subset \mathbb{R}^N$ has $\lambda(\Sigma)\leq \lambda_0$ and, for some $R_0\leq R\leq R_*$, we have that:
\begin{enumerate}
\item $B_R\cap \Sigma$ is a normal graph of $U$ over some cylinder in $\mathcal{C}_n(\mathring{\Gamma})$, with $\|U\|_{C^{2}(B_R)} \leq \epsilon_2$ and $\|U\|_{L^2(B_R)}^2 \leq \e^{-R^2/8}$,
\item $\|\phi\|_{L^2(B_R\cap\Sigma)}^2 \leq C \e^{-R_*^2/2}$,
\item $|\nabla^j A| \leq C_j$ on $B_R\cap \Sigma$ for all $j\leq l$;
\end{enumerate}
then there is a cylinder $\Gamma \in\mathcal{C}_n(\mathring{\Gamma})$ and a compactly supported normal vector field $V$ over $\Gamma$ with $\|V\|_{C^{2,\alpha}} \leq \epsilon_1$, such that $\Sigma \cap B_{R/(1+\theta)}$ is contained in the graph of $V$, and \[\|V\|_{L^2(B_{R/(1+\theta)})}^2 \leq \e^{-\frac{R^2}{4(1+\theta)^2}} .\] 
\end{theorem}
\begin{proof}
Theorem \ref{thm:lojasiewicz-intro} gives the desired graph $V$, and it remains only to estimate its $L^2$ norm. Since $R\leq R_*$ we have $C(q)^{-1} \|\phi\|_{L^q} \leq \|\phi\|_{L^2} \leq \e^{-R_*^2/4} \leq \e^{-R^2/4}$. Then we have

\begin{equation}\|V\|_{L^2}^2 \leq C\left(  \e^{-\frac{R^2}{4}}+ \e^{-\beta \frac{R^2}{2}}+ (R-5)^{\beta n} \e^{-\beta\frac{(R-5)^2}{4}}\right).
\end{equation}

Clearly $q\in(1,\frac{4}{3}), \beta\in(\frac{1}{q},1)$
may be chosen so that $\beta> \frac{1}{(1+\theta)^2}$. Then for $R_0$ sufficiently large we have $\|V\|_{L^2}^2 \leq \e^{-\frac{R^2}{4(1+\theta)^2}}$ as desired. 
\end{proof}

\section{Applications}
\label{sec:applications}

In this section we describe how the {\L}ojasiewicz inequality Theorem \ref{thm:lojasiewicz} and specifically the improvement step Theorem \ref{thm:improvement} imply the uniqueness and rigidity Theorems \ref{thm:unique} and \ref{thm:rigidity}. These follow the iterative extension/improvement method developed in \cite{CIM, CM15, CM19}; for completeness we state some key points more concretely and sketch the arguments. 

Recall that $\Sigma_s$ is a rescaled mean curvature flow (RMCF) if $(\pr_s X)^\perp = \phi$. Given a mean curvature flow $M_t$, there is an associated RMCF given by $\Sigma_s = e^{s/2} M_t$, where $s=-\ln(-t)$. Tangent flows to $M_t$ at $(0,0)$ correspond to the $s\to\infty$ limit of $\Sigma_s$. 

Given a submanifold $\Sigma$, the shrinker scale $R_\Sigma$ is defined by $\e^{-R_\Sigma^2/2} = \|\phi\|^2_{L^2(\Sigma)}$. Given a RMCF $\Sigma_s$, the shrinker scale $R_T$ is defined by \[\e^{-R_T^2/2} = \int_{T-1}^{T+1} \|\phi\|^2_{L^2(\Sigma_s)} \d s = F(\Sigma_{T-1}) - F(\Sigma_{T+1}).\] In both cases we understand the shrinker scale to be $\infty$ if the right hand side vanishes. 

\subsection{Extension step}

The following `extension step' is implied by the arguments of Colding-Minicozzi; which extends the domain of a graph by a multiplicative factor:

\begin{proposition}[\cite{CM15, CM19}]
\label{prop:extension}

Let $\mathring{\Gamma}$ be a closed shrinker, and take $K=\infty$ if $\mathring{\Gamma}$ is embedded or $K<\infty$ otherwise. Given $\epsilon_2>0$ and $\lambda_0$, there exist $\epsilon_3, R_2 ,C, C_l,\mu>0$ so that the following holds: Let $\Sigma_s$ be a RMCF with $\sup_s \sup_{\Sigma_s} |A_{\Sigma_s}|^2 \leq K$ and entropy at most $\lambda_0$, and suppose $R\in [R_2,R_T]$ is such that for any $s\in [T-\frac{1}{2},T+1]$, $B_R\cap \Sigma_s$ is the graph of $U$ over a fixed cylinder $\Gamma \in \mathcal{C}_n(\mathring{\Gamma}) $ with $\|U\|_{C^{2,\alpha}(B_R)}\leq \epsilon_3$. Then for $s\in [T-\frac{1}{2},T+1]$ we have:

\begin{enumerate}
\item $B_{(1+\mu)R}\cap \Sigma_s$ is contained in the graph of some extended $U$ with $\|U\|_{C^{2,\alpha}(B_{(1+\mu)R})}\leq \epsilon_2$;
\item $\|\phi\|^2_{L^2(B_{(1+\mu)R}\cap \Sigma_s)} \leq C \e^{-R_T^2/2}$;
\item $|\nabla^l A| \leq C_l$ on $B_{(1+\mu )R} \cap \Sigma_s$ for each $l$. 
\end{enumerate}
Moreover, the extension satisfies $\|U\|_{L^2(B_{(1+\mu)R})}^2 \leq \|U\|_{L^2(B_R)}^2 + C_n\epsilon_2 \lambda(\Gamma) R^{n-2} \e^{-R^2/4}$.
\end{proposition}
\begin{proof}
The results in \cite[Section 5.2-5.3]{CM15} are stated for general RMCF; the only remaining ingredient \cite[Lemma 5.39]{CM15} is a short-time stability that holds for any shrinker with uniformly bounded geometry. As pointed out in \cite[Section 7.2]{CM19}, these results continue to hold in high codimension, so the proposition follows from the formal argument in \cite[Section 5.5]{CM15}. Note that White's version \cite{W05} of Brakke's regularity theorem (see also \cite{ecker}) is used to obtain curvature bounds when $\Gamma$ is embedded; this step is not needed in the cases when one assumes a global curvature bound.

The last assertion follows from the cutoff Lemma \ref{lem:cutoff} as in the proof of Theorem \ref{thm:lojasiewicz}:
\[
\begin{split}
\|U\|_{L^2(B_{(1+\mu)R})}^2 &= \int_{\Gamma\cap B_{R}} |U|^2 \rho +  \int_{\Gamma\cap B_{(1+\mu)R} \setminus B_{R}} |U|^2 \rho
\\&\leq  \|U\|_{L^2(B_{R})}^2 + C_n\epsilon_2 \lambda(\Gamma) R^{n-2} \e^{-R^2/4}.
\end{split}
\]
\end{proof}

The reader may also consult \cite{zhu2023uniqueness} for details of a very similar extension lemma. In the special case of the static RMCF generated by a shrinker $\Sigma$, we have (cf. \cite{CIM}):

\begin{corollary}
\label{cor:extension-static}
Let $\mathring{\Gamma}$ be a closed shrinker, and take $K=\infty$ if $\mathring{\Gamma}$ is embedded or $K<\infty$ otherwise. Given $\epsilon_2>0$ and $\lambda_0$, there exist $\epsilon_3, R_2 ,C_l,\mu>0$ so that the following holds: Let $\Sigma$ be a shrinker with entropy at most $\lambda_0$ and $\sup_\Sigma |A_\Sigma|^2 \leq K$, and suppose $R\geq R_2$ is such that $B_R\cap \Sigma$ is the graph of $U$ over a fixed cylinder $\Gamma \in \mathcal{C}_n(\mathring{\Gamma}) $ with $\|U\|_{C^{2,\alpha}(B_R)}\leq \epsilon_3$. Then:

\begin{enumerate}
\item $B_{(1+\mu)R}\cap \Sigma$ is contained in the graph of some extended $U$ with $\|U\|_{C^{2,\alpha}(B_{(1+\mu)R})}\leq \epsilon_2$;
\item $|\nabla^l A| \leq C_l$ on $B_{(1+\mu)R}\cap \Sigma $ for each $l$. 
\end{enumerate}

\end{corollary}

\subsection{Uniqueness of blowups}

The crucial point for uniqueness of blowups is to establish a `discrete differential inequality' for the $F$ functional, with exponent $\frac{1+\mu}{2}>\frac{1}{2}$. (See \cite[Theorem 6.1]{CM15}, \cite[Theorem 7.1]{CM19}.) This again follows by the extension-improvement iteration. 

\begin{theorem}
\label{thm:final-ext}
Let $\mathring{\Gamma}$ be a closed shrinker satisfying (A1-A2) and $B_1(\mathring{\Gamma})\neq 0$, and take $K_0=\infty$ if $\mathring{\Gamma}$ is embedded or $K_0<\infty$ otherwise. Given $\epsilon_0>0$ and $\lambda_0$, there exists $R_1,\epsilon, K>0$ and $\bar{\mu}\in(0,\frac{1}{3})$ such that if $\Sigma_s$ is a RMCF with $\sup_s \sup_{\Sigma_s} |A_{\Sigma_s}|^2 \leq K_0$, entropy at most $\lambda_0$, and if, for $s\in[T-1,T+1]$, $B_{R_1}\cap \Sigma_s$ is a graph $U$ over a cylinder in $\mathcal{C}_n(\mathring{\Gamma})$ with $\|U\|_{C^{2,\alpha}}\leq \epsilon$, then:
\begin{enumerate}
\item There is a cylinder $\Gamma \in \mathcal{C}_n(\mathring{\Gamma})$ and a compactly supported normal field $V$ on $\Gamma$ so that $B_{(1+\bar{\mu})R_T} \cap \Sigma_T$ is contained in the graph of $V$, which satisfies $\|V\|_{C^{2,\alpha}} \leq \epsilon_0$, $\|V\|_{L^2}^2 \leq \e^{-(1-\bar{\mu})R_T^2/4}$ and $\|\phi_V\|_{L^2}^2 \leq \e^{-(1+3\bar{\mu})\frac{R_T^2}{4}} $;
\item For each $l$, there exists $C_l$ so that $\sup_{B_{(1+\bar{\mu})R_T\cap \Sigma_s}} |\nabla^l A| \leq C_l$ for $s\in [T-\frac{1}{2},T+1]$.
\item $|F(\Sigma_T) - F(\Gamma)| \leq K(F(\Sigma_{T-1})-F(\Sigma_{T+1}))^\frac{1+\bar{\mu}}{2}.$
\end{enumerate}
\end{theorem}
\begin{proof}
The proof essentially proceeds as in \cite[Section 7]{CM19}. 

Fix $\epsilon_2$ as in Theorem \ref{thm:improvement}, then take $\epsilon_3$ as determined by Proposition \ref{prop:extension} with this $\epsilon_2$. We will then work with $R_1 \geq \max(R_0,R_2)$, where $R_0$ is determined by Theorem \ref{thm:improvement} with $\epsilon_1= \frac{1}{2}\epsilon_3$, and $R_2$ is as in Proposition \ref{prop:extension}. Also fix $\mu \in(0, \sqrt{2}-1)$ smaller than that given by Proposition \ref{prop:extension}; Theorem \ref{thm:improvement} allows us to fix $\theta\in(0,\mu)$ in a way to be determined later.

The iterative hypothesis at scale $R$ is: 
\begin{itemize}
\item[($\star$)] There is $\Gamma \in \mathcal{C}_n(\mathring{\Gamma})$ so that for each $s\in [T-\frac{1}{2},T+1]$, $B_R \cap\Sigma_s$ is a normal graph of $U$ over $\Gamma$, with $\|U\|_{L^2(B_R)}^2 \leq 2\e^{-R^2/4}$ and $\|U\|_{C^{2,\alpha}(B_R)}\leq \epsilon_3$.
\end{itemize}

Note that, for small enough $\epsilon<\epsilon_1$, ($\star$) will certainly hold at $R=R_1$. Given $R \in [R_1, R_T]$ for which ($\star$) holds, applying the extension step Proposition \ref{prop:extension} gives extended $U$ on $B_{(1+\mu)R}$ satisfying $\|U\|_{C^2(B_{(1+\mu)R})} \leq \epsilon_2$ and, if $R_1$ is large enough, (note $(1+\mu)^2<2$)
\begin{equation}
\label{eq:extended-U}
\|U\|_{L^2(B_{(1+\mu)R})}^2 \leq 2\e^{-R^2/4} + C_n\lambda(\Gamma)R^{n-2} \e^{-R^2/4}\leq \e^{-(1+\mu)^2 R^2/8}.
\end{equation}

If $(1+\mu)R \leq R_T$, then we may apply the improvement step to show that ($\star$) holds at scale $\frac{1+\mu}{1+\theta}R$. Indeed, applying Theorem \ref{thm:improvement} (at scale $R\mapsto (1+\mu)R$ and with $R_*= R_T$) at any $s$ gives that $B_{\frac{1+\mu}{1+\theta}R}\cap \Sigma_s$ is contained in a normal graph of $V$, with the estimates $\|V\|_{L^2(B_{\frac{1+\mu}{1+\theta}R})}^2 \leq \e^{-\left(\frac{1+\mu}{1+\theta}\right)^2\frac{R^2}{4^2}}$ and $\|V\|_{C^{2,\alpha}} \leq \epsilon_1$. To ensure that we use the same cylinder for all $s$, we first take this normal field $V$ only at $s=T-\frac{1}{2}$ and fix the corresponding cylinder $\Gamma$. Now fix $l>3$ so that $a_l = a_{l,n,3} > \frac{1}{(1+\theta)^2}$; by interpolation (using conclusions (2-3) of Proposition \ref{prop:extension}) we have $\sup_{s\in[T-\frac{1}{2}, T+1]} \sup_{B_{(1+\mu)R-1}\cap \Sigma_s} |\phi|_3 \leq C\e^{-a_l R_T^2/4}$. Since $\phi$ is the velocity of the RMCF, it follows that for any $s\in [T-\frac{1}{2},T+1]$, $B_{\frac{1+\mu}{1+\theta}R}\cap \Sigma_s$ is contained in a graph (still denoted $V$) over the fixed cylinder $\Gamma$ with the desired estimates \[\|V\|_{L^2(B_{\frac{1+\mu}{1+\theta}R})}^2 \leq \e^{-\left(\frac{1+\mu}{1+\theta}\right)^2\frac{R^2}{4^2}} +C\e^{-a_l R_T^2/4} \leq 2\e^{-\left(\frac{1+\mu}{1+\theta}\right)^2\frac{R^2}{4^2}},\] \[\|V\|_{C^{2,\alpha}} \leq \epsilon_1 + C\e^{-a_l R_T^2/4} \leq \epsilon_3, \] where we have used that $(1+\mu)R\leq R_T$ are large. 

By iterating these steps, we can reach a scale $R$ for which $(\star)$ holds, but $(1+\mu)R > R_T$. In this case, we can still use the extension step, and (\ref{eq:extended-U}) implies that $\|U\|_{L^2(B_{R_T})}^2 \leq \e^{-R_T^2/8}$. Therefore, we may apply the improvement step at scale $R_T$ (with $R_*=R_T$), which gives that $(\star)$ holds at scale $\frac{1}{1+\theta}R_T$. We then apply the extension step one final time, extending $V$ to scale $\frac{1+\mu}{1+\theta}R_T$, so that $B_{\frac{1+\mu}{1+\theta}R_T}\cap \Sigma_s$ is contained in a graph of $V$, where $\|V\|_{C^{2,\alpha}} \leq \epsilon_2$ and 
\[\|V\|_{L^2(B_{\frac{1+\mu}{1+\theta}R_T})}^2 \leq 2\e^{-\frac{1}{(1+\theta)^2}R_T^2/4}+C_n\epsilon_2 \lambda(\Gamma)R_T^{n-2} \e^{-\frac{1}{(1+\theta)^2}R_T^2/4}.
\]

By item (2) of Proposition \ref{prop:extension}, we have $\|\phi\|_{B_{\frac{1+\mu}{1+\theta}R_T}\cap\Sigma_s} \leq C\e^{-R_T^2/2}$, so by Lemma \ref{lem:cutoff} again,  \[\|\phi_V\|_{L^2}^2 \leq C\e^{-R_T^2/2} + CR_T^n \e^{-\left(\frac{1+\mu}{1+\theta}\right)^2R_T^2/4} 
\] 

Choose $\theta$ small enough so that, for some $\bar{\mu}\in(0,\frac{1}{3})$, \[ 1-\bar{\mu} < \frac{1}{(1+\theta)^2}, \qquad 1+\bar{\mu} < \frac{1+\mu}{1+\theta}, \qquad 1+3\bar{\mu} < \frac{(1+\mu)^2}{(1+\theta)^2} .\] Then $\|V\|_{L^2}^2\leq \e^{-(1-\bar{\mu})R_T^2/4}$ and $\|\phi_V\|_{L^2}^2\leq \e^{-(1+3\bar{\mu}) R_T^2/4}$. Consequently,
for the $F$ functional, by Proposition \ref{prop:F-exp}, it follows that (note $3(1-\bar{\mu}) > 1+\bar{\mu}$) \[|F(\Gamma_V) - F(\Gamma)| \leq C(\|V\|_{L^2} \|\phi_V\|_{L^2} + \|V\|_{L^2}^3) \leq Ce^{-(1+\bar{\mu})\frac{R_T^2}{4}}.\]

Using Lemma \ref{lem:cutoff} once more we conclude that 
\[
\begin{split}
|F(\Sigma_T)-F(\Gamma)| &\leq  |F(\Gamma_V)-  F(\Gamma)|+ CR_T^{n-2}\e^{-\left(\frac{1+\mu}{1+\theta}\right)^2 R_T^2/4} 
\\& \leq K \e^{-(1+\bar{\mu})R_T^2/4} = K ( F(\Sigma_{T-1}) - F(\Sigma_{T+1}))^{\frac{1+\bar{\mu}}{2}}.
\end{split}
\]
\end{proof}

The proof of uniqueness proceeds as in Colding-Minicozzi \cite[Theorem 0.2]{CM15} (see also \cite[Theorem 0.1]{CM19}), using item (3) above for the discrete differential inequality. Again for completeness, we sketch the main points as follows:

\begin{proof}[Proof of Theorem \ref{thm:unique}]
Without loss of generality we may assume $(x_0,t_0)=(0,0)$ and consider the RMCF $\Sigma_s = e^{s/2} M_t$, $s=-\ln(-t)$ as above. Arguing exactly as for \cite[Corollary 0.3]{CIM}, the rigidity Theorem \ref{thm:rigidity} implies that any other tangent flow at $(x_0,t_0)$ must be induced by a cylinder in $\mathcal{C}_n(\mathring{\Gamma})$. It follows that:
\begin{itemize}
\item[($\dagger$)] For all $T\geq s_0$, there is a cylinder $\Gamma' \in \mathcal{C}_n(\mathring{\Gamma})$ so that for all $s\in [T-1,T+1]$, $\Sigma_s \cap B_{R_1}$ is a normal graph over $\Gamma'$, with $C^{2,\alpha}$ norm at most $\epsilon_1$. 
\end{itemize}

In particular, Theorem \ref{thm:final-ext} applies for all $T\geq s_0$. By item (3) of that theorem, and \cite[Lemma 6.9]{CM15}, we have that $\sum_{j=1}^\infty (F(\Sigma_j)-F(\Sigma_{j+1}))^\frac{1}{2} <\infty$. This implies the uniqueness, noting that since $\phi$ is the velocity of RMCF, the $L^1$ distance between time slices is at most \[\int_{s_1}^{s_2}\|\phi\|_{L^1(\Sigma_s)}\d s \leq \sqrt{F(\Sigma_0)} (F(\Sigma_{s_1}) - F(\Sigma_{s_2}))^\frac{1}{2}.\] 

Note that for the initial closeness ($\dagger$) and the uniqueness of tangent flow type, one passes to smooth convergence, using White's version of Brakke regularity if $\mathring{\Gamma}$ is embedded, or using the type I assumption (which implies a uniform curvature bound on $\Sigma_s$) otherwise. 
\end{proof}

\subsection{Rigidity}

The rigidity follows by combining Theorem \ref{thm:improvement} with Corollary \ref{cor:extension-static}:

\begin{proof}[Proof of Theorem \ref{thm:rigidity}]
For large enough $R_1$ we will have $\|U\|_{L^2(B_{R_1})}^2 \leq \e^{-R_1^2/4}$. Similar to the proof of Theorem \ref{thm:final-ext}, we may then iterate the extension step Corollary \ref{cor:extension-static} with the improvement step Theorem \ref{thm:improvement} (which may be applied with $R_*$ arbitrarily large, as $\phi_\Sigma\equiv0$). The iteration shows that for arbitrarily large $R$, $\Sigma\cap B_R$ is a graph $V$ over some $\Gamma\in \mathcal{C}_n(\mathring{\Gamma}) $ with $\|V\|_{L^2}^2 \leq \e^{-R^2/4}\to 0$, hence $\Sigma$ must be in $\mathcal{C}_n(\mathring{\Gamma})$. 
\end{proof}

\appendix

\section{Interpolation}
\label{sec:interpolation}

We record the consequences of interpolation for Gaussian weight. In this appendix, $L^p$ refers to the unweighted space, with $L^p_\rho$ referring to the weighted space. 

As in \cite[Appendix B]{CM15}, the statements in this appendix hold equally well for tensor quantities on a manifold with uniformly bounded geometry, of course with different constants. 

We recall the following interpolation result from \cite{CM15}, which is a special case of the Gagliardo-Nirenberg interpolation inequality (cf. \cite[Theorem 1.5.2]{ChM} or \cite[Section 5.6]{evans}): 

\begin{lemma}
There exists $C=C(m,j,n)$ so that if $u$ is a $C^m$ function on $B^n_{r}$, then for $j\leq m$, setting $a_{m,j,n} = \frac{m-j}{m+n}$ we have
\[r^j \|\nabla^j u\|_{L^\infty(B_r)} \leq C\left( r^{-n} \|u\|_{L^1(B_{r})} + r^j \|u\|_{L^1(B_{r})}^{a_{m,j,n}} \|\nabla^m u \|^{1-a_{m,j,n}}_{L^\infty(B_{r})}\right).\]
\end{lemma}

For our purposes we consider quantities on a generalised cylinder $\Gamma = \mathring{\Gamma}^k \times \mathbb{R}^{n-k}$, where $\mathring{\Gamma}$ is compact, and recall the Gaussian weight $\rho = (4\pi)^{-n/2} \e^{-|x|^2/4}$. We apply interpolation at scale $r= \frac{1}{1+|x|}$. This gives \[\| \nabla^j u\|_{L^\infty(B_r(x))} \leq C\left( \langle x\rangle^{n+j} \|u\|_{L^1(B_{r}(x))} + \|u\|_{L^1(B_{r}(x))}^{a_{m,j,n}} \|\nabla^m u \|^{1-a_{m,j,n}}_{L^\infty(B_r(x))}\right).\]

Working at scale $r=\frac{1}{1+|x|}$ means that the variation of $\log\rho$ is uniformly bounded on $B_r(x)$. So for $x\in B_{R-1}$, $p\geq 1$ we can estimate $\|u\|_{L^1(B_{r}(x))} \leq C(p) \|u\|_{L^p(B_{r}(x))}\leq C' \e^{\frac{1}{4p}|x|^2} \|u\|_{L^p_\rho(B_R)}$. Then defining $M_j = \|\nabla^j u\|_{L^\infty(B_R)}$, we have
\begin{equation}
\label{eq:interp-est}
|\nabla^j u|(x)\leq C(m,n,j,p)\left(\langle x\rangle^{n+j} M_0^{1-a_{m,j,n}} + M_m^{1-a_{m,j,n}}  \right) \e^{\frac{a_{m,j,n} |x|^2}{4p}} \|u\|_{L^p_\rho(B_R)}^{a_{m,j,n}}.
\end{equation}

Since $\langle x\rangle^b \e^{\frac{a}{4}|x|^2}$ is $\rho$-integrable for $a<1$, it follows that for sufficiently large $R$, \begin{equation}\|u\|_{W^{j,p}_\rho(B_{R-1}) }\leq C(m,j,n,p,M_0,M_m) \|u\|_{L^p_\rho(B_R)}^{a_{m,j,n}}.\end{equation} 

In fact, consider a fixed scale $r_0<1$. Then by (\ref{eq:interp-est}), 

\[\|\nabla^j u\|_{L^\infty(B_{r_0}(x))} \leq C \left(\langle x\rangle^{n+j} M_0^{1-a_{m,j,n}} + M_m^{1-a_{m,j,n}}  \right) \e^{\frac{a_{m,j,n} }{4p}(|x|+r_0)^2} \|u\|_{L^p_\rho(B_R)}^{a_{m,j,n}}. \]

So since $\langle x\rangle^b \e^{\frac{a}{4}(|x|+r_0)^2}$ is also $\rho$-integrable for $a<1$, it follows that
\begin{equation}
\label{eq:interp-est2}
 \left(\int_{B_{R-1}} \langle x\rangle^q \|u\|^p_{C^j(B_{r_0}(x))} \rho(x) \d x\right)^\frac{1}{p} \leq C(m,j,n,p,M_0,M_m,q, r_0) \|u\|_{L^p_\rho(B_R)}^{a_{m,j,n}}.\end{equation}

\bigskip

\bibliographystyle{plain}
\bibliography{taylor}

\end{document}